\documentclass[10pt,twoside,a4paper]{amsart}

\usepackage[export]{adjustbox}
\usepackage{graphicx,amsthm,amsmath,geometry,hyperref}
\usepackage{epstopdf}

\usepackage[usenames]{color}
\usepackage{amssymb}

\usepackage{amscd}

\usepackage{amsfonts}

\usepackage{enumitem}
\usepackage{stmaryrd}

\newtheorem{theorem}{Theorem}[section]

\newtheorem{proposition}[theorem]{Proposition}
\newtheorem{corollary}[theorem]{Corollary}
\theoremstyle{definition}

\theoremstyle{remark}
\newtheorem{remark}[theorem]{Remark}

\newcommand{\ar}{\mathcal{AR}}
\newcommand{\arr}{\mathcal{AR}_{a,b}^k}
\newcommand{\al}{\beta}
\newcommand{\be}{\alpha}
\newcommand{\ga}{\gamma}
\newcommand{\de}{\delta}

\newcommand{\oo}{\ooo(a)_{i,j}}
\newcommand{\mm}{\mathcal{M}}
\newcommand{\dk}{\ar^{k-1,j}_{a,b}}
\newcommand{\pp}{\prime}
\newcommand{\dkk}{\ar_{a,b-1}^{k-2,j-1}}
\newcommand\hyper[2]{\genfrac{}{}{0pt}{}{#1}{#2}}

\DeclareMathOperator{\pf}{Pf}
\DeclareMathOperator{\m}{M}
\DeclareMathOperator{\ad}{AD}
\DeclareMathOperator{\vv}{V}

\DeclareMathOperator{\ooo}{O}
\DeclareMathOperator{\opp}{OPP}
\DeclareMathOperator{\adj}{ADJ}
\DeclareMathOperator{\hex}{H}

\numberwithin{equation}{section}

\newcommand{\abs}[1]{\lvert#1\rvert}

\begin{document}

\title[Aztec Rectangles with boundary defects]{Enumeration of Domino Tilings of an Aztec Rectangle with boundary defects}

\author[M. P. Saikia]{Manjil P. Saikia}

\address{Universit\"at Wien, Fakult\"at f\"ur Mathematik, Oskar-Morgenstern-Platz 1, 1090 Wien, Austria} 
\email{manjil.saikia@univie.ac.at}

\thanks{Supported by the Austrian Science Foundation FWF, START grant Y463.}



\begin{abstract}
Helfgott and Gessel gave the number of domino tilings of an Aztec Rectangle with defects of size one on the boundary of one side. In this paper we extend this to the case  
of domino tilings of an Aztec Rectangle with defects on all boundary sides.
\end{abstract}

\maketitle

\section{Introduction}

Elkies, Kuperberg, Larsen and Propp in their paper \cite{diamond} introduced a new class of object which they called Aztec Diamonds. The Aztec Diamond of order $n$ (denoted by $\ad(n)$) is the union 
of all unit squares inside the contour $\abs{x}+\abs{y}=n+1$ (see Figure \ref{fig:diamond} for an Aztec Diamond of order $3$). A domino is the union of any two unit squares sharing an edge, and a domino tiling of a region is a covering of the region by dominoes so that there are no gaps or overlaps. The authors in \cite{diamond} and \cite{diamond2}
considered the problem of counting the number of domino tiling the Aztec Diamond with dominoes and presented four different proofs of the following result.

\begin{figure}[!htb]
\centering
\includegraphics[scale=.7]{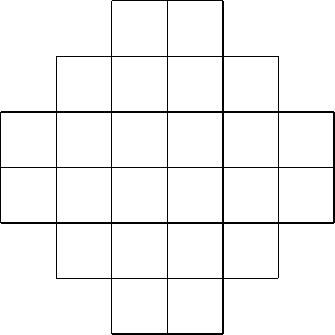}
\caption{$\ad(3)$, Aztec Diamond of order $3$}
\label{fig:diamond}
\end{figure}

\begin{theorem}[Elkies--Kuperberg--Larsen--Propp, \cite{diamond, diamond2}]\label{adm}

The number of domino tilings of an Aztec Diamond of order $n$ is $2^{n(n+1)/2}$.

\end{theorem}

This work subsequently inspired lot of follow ups, including the natural extension of the Aztec Diamond to the Aztec rectangle (see Figure \ref{fig:check}). We denote by $\ar_{a,b}$ 
the Aztec rectangle which has $a$ unit squares on the southwestern side and $b$ unit squares on the northwestern side. In the remainder of this paper, we assume $b\geq a$ unless 
otherwise mentioned. For $a<b$, $\ar_{a,b}$ does not have any tiling by dominoes. The non-tileability of the region $\ar_{a,b}$ becomes evident if we look at the checkerboard representation of $\ar_{a,b}$ (see Figure \ref{fig:check}). 
However, if we remove $b-a$ unit squares from the southeastern side then we have a simple product formula found by Helfgott and Gessel \cite{gessel}.

\begin{figure}[!htb]
\centering
\includegraphics[scale=.7]{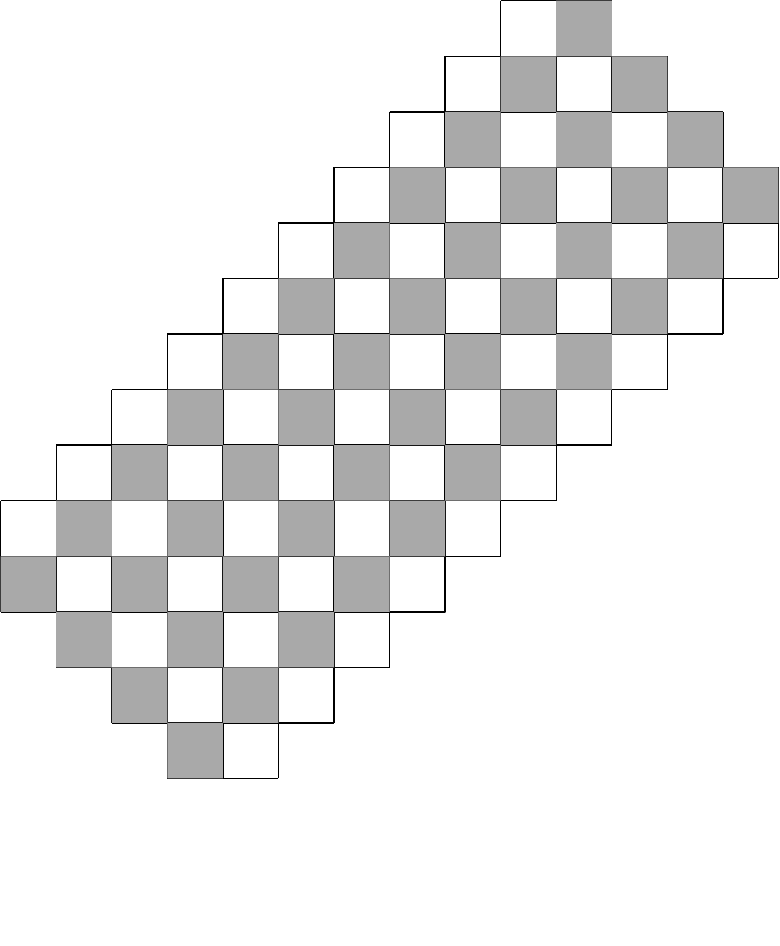}
\caption{Checkerboard representation of an Aztec Rectangle with $a=4, b=10$}
\label{fig:check}
\end{figure}

\begin{theorem}[Helfgott--Gessel, \cite{gessel}]\label{ar}
Let $a<b$ be positive integers and $1\leq s_1<s_2<\cdots <s_a\leq b$. Then the number of domino tilings of $\ar_{a,b}$ where all unit squares from the southeastern side are 
removed except for those in positions $s_1, s_2, \ldots, s_a$ is \[2^{a(a+1)/2}\prod_{1\leq i<j\leq a}\frac{s_j-s_i}{j-i}.\]
\end{theorem}

Tri Lai \cite{lai} has recently generalized Theorem \ref{ar} to find a generating function, following the work of Elkies, Kuperberg, Larsen and Propp \cite{diamond, diamond2}. Motivated by the recent work of Ciucu and Fischer \cite{ilse}, here we look at the problem of tiling an Aztec rectangle 
with dominoes if arbitrary unit squares are removed along the boundary of the Aztec rectangle.

This paper is structured as follows: in Section \ref{s2} we state our main results, in Section \ref{cond-sec} we introduce our main tool in the proofs and present a 
slight generalization of it, in Section \ref{s3} we look at tilings of some special cases which are used in our main results. Finally, in Section \ref{s4} we
 prove the results described in Section \ref{s2}. The main ingredients in most of our proofs will be the method of condensation developed by Kuo \cite{kuo} and its subsequent generalization by 
 Ciucu \cite{ciucu}.

\section{Statements of Main Results}\label{s2}

In order to create a region that can be tiled by dominoes we have to remove $k$ more white squares than black squares along the boundary of $\ar_{a,b}$. There are $2b$ white squares and $2a$ black squares on the boundary of $\ar_{a,b}$. We choose 
$n+k$ of the white squares that share an edge with the boundary and denote them by $\al_1, \al_2, \ldots, \al_{n+k}$ (we will refer to them as defects of type $\al$). We choose any 
$n$ squares from the black squares which share an edge with the boundary and denote them by $\be_1, \be_2, \ldots, \be_n$ (we refer to them as defects of type $\be$). We consider 
regions of the type $\ar_{a,b}\setminus \{\al_1, \ldots, \al_{n+k}, \be_1, \ldots, \be_n\}$, which are more general than the type considered in \cite{gessel}.

It is also known that domino tilings of a region can be identified with perfect matchings of its planar dual graph, so for any region $R$ on the square lattice we denote by $\m (R)$ the number of domino tilings 
of $R$. We now state the main results of this paper below. The first result is concerned with the case when the defects are confined to three of the four sides of the Aztec 
rectangle (defects do not occur on one of the sides with shorter length), and provides a Pfaffian expression for the number of tilings of such a region, with each entry in the Pfaffian being given by a simple product or by a sum or product of quotients of factorials and powers of $2$. The second result gives a nested Pfaffian expression for the general case when we do not restrict the occurence of defects 
on any boundary side. The third result deals with the case of an Aztec Diamond with arbitrary defects on the boundary and gives a Pfaffian expression for the number of tilings of such a 
region, with each entry in the Pfaffian being given by a simple sum of quotients of factorials and powers of $2$.

We define the region $\ar_{a,b}^k$ to be the region obtained from $\ar_{a.b}$ by adding a string of $k$ unit squares along the boundary of the southeastern side as shown in Figure 
\ref{fig:mt1}. We denote this string of $k$ unit squares by $\ga_1, \ga_2, \ldots, \ga_k$ and refer to them as defects of type $\ga$.

\begin{figure}[!htb]
\centering
\includegraphics[scale=.7]{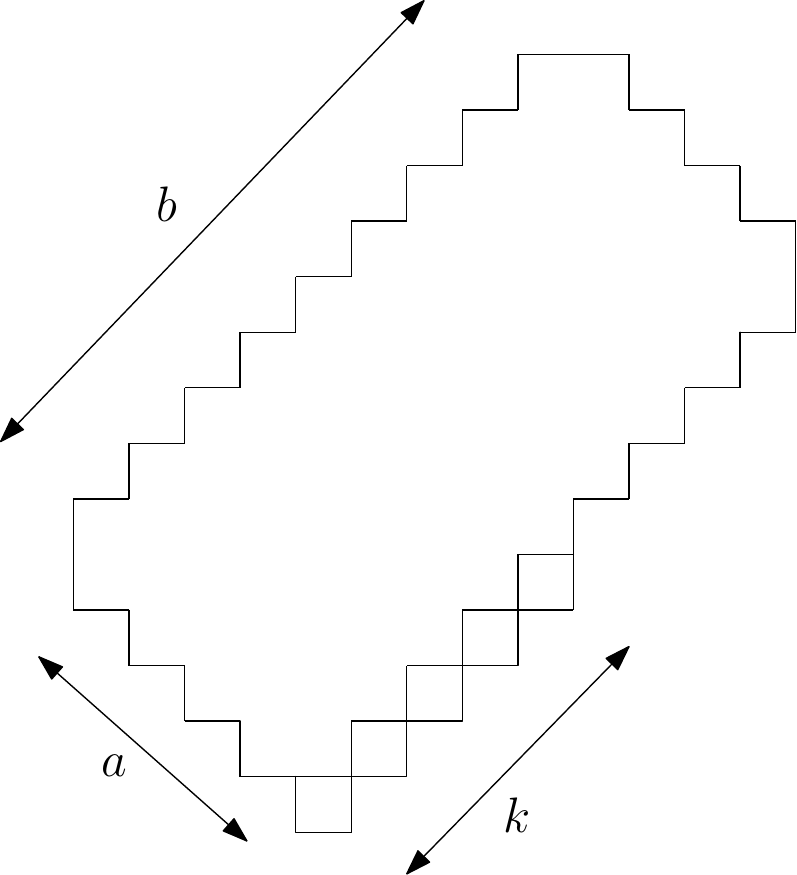}
\caption{$\ar_{a,b}^k$ with $a=4,b=9,k=5$}
\label{fig:mt1}
\end{figure}

\begin{theorem}\label{mt1}
 Assume that one of the two sides on which defects of type $\be$ can occur does not actually have any defects on it. Without loss of generality, we assume this to be the southwestern side. 
 Let $\de_1, \ldots, \de_{2n+2k}$ be the elements of the set $\{\al_1, \ldots, \al_{n+k}\}\cup \{\be_1, \ldots, \be_{n}\}\cup \{\ga_1, \ldots, \ga_{k}\}$ listed in a cyclic order.
 
 Then we have
 \begin{equation}\label{em1}
  \m (\ar_{a,b}\setminus \{\al_1, \ldots, \al_{n+k}, \be_1, \ldots, \be_n\})=\frac{1}{[\m (\ar_{a,b}^k)]^{n-k+1}}\pf[(\m (\ar_{a,b}^k\setminus\{\de_i, \de_j\}))_{1\leq i<j\leq 2n+2k}],
 \end{equation}

 \noindent where all the terms on the right hand side are given by explicit formulas:
 
 \begin{enumerate}
  \item $\m (\ar_{a,b}^k)$ is given by Theorem \ref{adm},
  \item $\m (\ar_{a,b}^k\setminus\{\al_i, \be_j\})$ is given by Proposition \ref{ad_i_j} if $\al_i$ is on the south-eastern side and not above a $\ga$ defect; otherwise it is $0$,
  \item $\m (\ar_{a,b}^k\setminus\{\al_i, \ga_j\})$ is given by Theorem \ref{adm} if $\al_i$ is above a $\ga$ defect; it is given by Proposition \ref{ar_k-1_i} if the $\al$ defect is in the northwestern side at a distance of more than $k-1$ from the western corner; it is given by Propositions \ref{ar_k_i} if 
  the $\al$ dent is on the southeastern side; otherwise it is $0$,
  \item $\m (\ar_{a,b}^k\setminus\{\al_i, \al_j\})=\m (\ar_{a,b}^k\setminus\{\be_i, \be_j\})=\m (\ar_{a,b}^k\setminus\{\be_i, \ga_j\})=\m (\ar_{a,b}^k\setminus\{\ga_i, \ga_j\})=0$.
 \end{enumerate}

\end{theorem}

\begin{theorem}\label{mt2}
 
 Let $\al_1, \ldots, \al_{n+k}$ be arbitrary defects of type $\al$ and $\be_1, \ldots, \be_n$ be arbitrary defects of type $\be$ along the boundary of $\ar_{a,b}$. Then $\m ({\ar_{a,b}}\setminus\{\al_1, \ldots, \al_{n+k}, \be_1, \ldots, \be_n\})$ is equal to the Pfaffian of a $2n\times 2n$ 
 matrix whose entries are Pfaffians of $(2k+2)\times (2k+2)$ matrices of the type in the statement of Theorem \ref{mt1}.
 
\end{theorem}

In the special case when the number of defects of both types are the same, that is when $k=0$ we get an Aztec Diamond with arbitrary defects on the boundary and the number of tilings can be given by a Pfaffian 
where the entries of the Pfaffian are explicit, as stated in the theorem below.

\begin{theorem}\label{mt3}
 Let $\al_1, \ldots, \al_{n}$ be arbitrary defects of type $\al$ and $\be_1, \ldots, \be_n$ be arbitrary defects of type $\be$ along the boundary of $\ad(a)$, and let $\de_1, \ldots, \de_{2n}$ 
 be a cyclic listing of the elements of the set $\{\al_1, \ldots, \al_n\}\cup \{\be_1, \ldots, \be_n\}$. Then
 
 \begin{equation}\label{emt3}
  \m (\ad(a)\setminus\{\al_1, \ldots, \al_n, \be_1, \ldots, \be_n\})=\frac{1}{[\m (\ad(a))]^{n-1}}\pf [(\m (\ad(a)\setminus\{\de_i, \de_j\}))_{1\leq i<j\leq 2n}],
 \end{equation}

 \noindent where the values of $\m (\ad(a)\setminus\{\de_i, \de_j\}))$ are given explicitly as follows:
 
 \begin{enumerate}
  \item $\m (\ad(a)\setminus\{\al_i, \be_j\}))$ is given by Proposition \ref{ad_i_j},
  \item $\m (\ad(a)\setminus\{\al_i, \al_j\}))=\m (\ad(a)\setminus\{\be_i, \be_j\}))=0$.
 \end{enumerate}

\end{theorem}

\section{A result on Graphical Condensation}\label{cond-sec}

The proofs of our main results are based on Ciucu's generalization \cite{ciucu} of Kuo's graphical condensation \cite{kuo} which we state below. The aim of this section is also to present 
our small generalization of Ciucu's result.

Let $G$ be a weighted graph, where the weights are associated with each edge of $G$, and let $\m (G)$ denote the sum of the weights of the perfect matchings of $G$, where the weight of a perfect matching is taken to be the product 
of the weights of its constituent edges. We are interested in graphs with edge weights all equaling $1$, which corresponds to tilings of the region in our special case.

\begin{theorem}[Ciucu, \cite{ciucu}]\label{condensation}
Let $G$ be a planar graph with the vertices $a_1, a_2, \ldots, a_{2k}$ appearing in that cyclic order on a face of $G$. Consider the skew-symmetric matrix $A=(a_{ij})_{1\leq i,j\leq 2k}$ with entries given by 

\begin{equation}\label{ciucu1}
a_{ij} :=  \m (G\setminus \{a_i, a_j\}), \text{if } i<j.
\end{equation}

Then we have that
\begin{equation}\label{ciucu2}
\m (G\setminus \{a_1, a_2, \ldots, a_{2k}\})=\frac{\pf(A)}{[\m (G)]^{k-1}}.
\end{equation}
\end{theorem}

Although Theorem \ref{condensation} is enough for our purposes, we state and prove a slightly more general version of the theorem below. It turns out that our result is a common generalization for 
the condensation results in \cite{kuo} as well as Theorem \ref{condensation} which follows immediately from Theorem \ref{condensation-2} below if we consider $a_1, \ldots, a_{2k}\in \vv(G)$. We also mention 
that Corollary \ref{cond-cor} of Theorem \ref{condensation-2}, does not follow from Theorem \ref{condensation}.

To state and prove our result, we will need to make some notations and concepts clear. We consider the symmetric difference on the vertices and edges of a graph. Let $H$ be a planar graph and $G$ be an 
induced subgraph of $H$ and let $W\subseteq \vv(H)$. Then we define $G+W$ as follows: $G+W$ is the induced subgraph of $H$ with vertex set $\vv(G+W)=\vv(G)\Delta W$, where $\Delta$ 
denotes the symmetric difference of sets. Now we are in a position to state our result below.

\begin{theorem}\label{condensation-2}
Let $H$ be a planar graph and let $G$ be an induced subgraph of $H$ with the vertices $a_1, a_2, \ldots, a_{2k}$ appearing in that cyclic order on a face of $H$. Consider the skew-symmetric matrix $A=(a_{ij})_{1\leq i,j\leq 2k}$ with entries given by

\begin{equation}\label{ciucu1-2}
a_{ij} :=  \m (G+ \{a_i, a_j\}), \text{if } i<j.
\end{equation}

Then we have that
\begin{equation}\label{cond4}
\m (G+ \{a_1, a_2, \ldots, a_{2k}\})=\frac{\pf(A)}{[\m (G)]^{k-1}}.
\end{equation}
\end{theorem}

\begin{corollary}\cite[Theorem 2.4]{kuo}\label{cond-cor}
 Let $G=(V_1, V_2, E)$ be a bipartite planar graph with $\abs{V_1}=\abs{V_2}+1$; and let $w, x, y$ and $z$ be vertices of $G$ that appear in cyclic order on a face of $G$. If 
 $w, x, y \in V_1$ and $z\in V_2$ then 
 \begin{align}
  \m (G-\{w\})\m (G-\{x,y,z\})+\m (G-\{y\})\m (G-\{w, x,z\})=\m (G-\{x\})\m (G-\{w, y,z\}) & \\ 
  +\m (G-\{z\})\m (G-\{w, x,y\}). & \nonumber
 \end{align}
\end{corollary}

\begin{proof}
 Take $n=2$, $a_1=w, a_2=x, a_3=y, a_4=z$ and $G=H\setminus\{a_1\}$ in Theorem \ref{condensation-2}.
\end{proof}

The proof of Theorem \ref{condensation} follows from the use of some auxillary results. In the vein of those results, we need the following proposition to complete our proof of 
Theorem \ref{condensation-2}.

\begin{proposition}\label{ck3}
Let $H$ be a planar graph and $G$ be an induced subgraph of $H$ with the vertices $a_1, \ldots, a_{2k}$ appearing in that cyclic order among the vertices of some face of $H$. 
   Then
 
 \begin{align}\label{prope1}
  \m(G)\m(G+\{a_1, \ldots, a_{2k}\})+\sum_{l=2}^{k}\m(G+ \{a_1, a_{2l-1}\})\m(G+ \overline{\{a_1, a_{2l-1}\}}) & \nonumber \\
  = \sum_{l=1}^{k}\m(G+ \{a_1, a_{2l}\})\m(G+ \overline{\{a_1, a_{2l}\}}), & 
 \end{align}
 
 \noindent where $\overline{\{a_i, a_j\}}$ stands for the complement of $\{a_i, a_j\}$ in the set $\{a_1, \ldots, a_{2k}\}$.

\end{proposition}

Our proof follows closely that of the proof of an analogous proposition given by Ciucu \cite{ciucu}.
\begin{proof}
 We recast equation \eqref{prope1} in terms of disjoint unions of cartesian products as follows
 
  \begin{align}\label{prope2}
  \mm(G)\times \mm(G+\{a_1, \ldots, a_{2k}\})\cup \mm(G+ \{a_1, a_{3}\})\times \mm(G+ \overline{\{a_1, a_{3}\}})\cup \ldots & \nonumber \\
  \cup \mm(G+ \{a_1, a_{2k-1}\})\times \mm(G+ \overline{\{a_1, a_{2k-1}\}})& 
 \end{align}
 
 \noindent and
   \begin{align}\label{prope3}
 \mm(G+ \{a_1, a_{2}\})\times \mm(G+ \overline{\{a_1, a_{2}\}})\cup\mm(G+ \{a_1, a_{4}\})\times \mm(G+ \overline{\{a_1, a_{4}\}}) \cup \ldots & \nonumber \\
  \cup \mm(G+ \{a_1, a_{2k}\})\times \mm(G+ \overline{\{a_1, a_{2k}\}})\cup& 
 \end{align}

 \noindent where $\mm(F)$ denotes the set of perfect matchings of the graph $F$. For each element $(\mu, \nu)$ of \eqref{prope2} or \eqref{prope3}, we think of the edges of $\mu$ as being marked by solid lines and 
 that of $\nu$ as being marked by dotted lines, on the same copy of the graph $H$. If there are any edges common to both then we mark them with both solid and dotted lines.
 
 We now define the weight of $(\mu, \nu)$ to be the product of the weight of $\mu$ and the weight of $\nu$. Thus, the total weight of the elements in the set \eqref{prope2} is same as the left 
 hand side of equation \eqref{prope1} and the total weight of the elements in the set \eqref{prope3} equals the right hand side of equation \eqref{prope1}. To prove our result, we 
 have to construct a weight-preserving bijection between the sets \eqref{prope2} and \eqref{prope3}.
 
 Let $(\mu, \nu)$ be an element in \eqref{prope2}. Then we have two possibilities as discussed in the following. If $(\mu, \nu)\in \mm(G)\times \mm(G+\{a_1, \ldots, a_{2k}\})$ we note that when considering the edges of $\mu$ and $\nu$ together on the 
 same copy of $H$, each of the vertices $a_1, \ldots, a_{2k}$ is incident to precisely one edge (either solid or dotted depending on the graph $G$ and the vertices $a_i$'s), while all the other 
 vertices of $H$ are incident to one solid and one dotted edge. Thus $\mu \cup \nu$ is the disjoint union of paths connecting the $a_i$'s to one another in pairs, and cycles covering 
 the remaining vertices of $H$. We now consider the path containing $a_1$ and change a solid edge to a dotted edge and a dotted edge to a solid edge. Let this pair of matchings be $(\mu^\pp, \nu^\pp)$.
 
 The path we have obtained must connect $a_1$ to one of the even-indexed vertices, if it connected $a_1$ to some odd-indexed vertex $a_{2i+1}$ then it would isolate the $2i-1$ vertices $a_2, a_3, 
 \ldots, a_{2i}$ from the other vertices and hence we do not get disjoint paths connecting them. Also, we note that the end edges of this path will be either dotted or solid 
 depending on our graph $G$ and the vertices $a_i$'s. So $(\mu^\pp, \nu^\pp)$ is an element of \eqref{prope3}.
 
 If $(\mu, \nu)\in \mm(G+\{a_1, a_3\})\times \mm(G+ \overline{\{a_1, a_{3}\}})$, then we map it to a pair of matchings $(\mu^\pp, \nu^\pp)$ obtained by reversing the solid and dotted edges 
 along the path in $\mu\cup \nu$ containing $a_3$. With a similar reasoning like above, this path must connect $a_3$ to one of the even-indexed vertices and a similar argument will show 
 that indeed $(\mu^\pp, \nu^\pp)$ is an element of \eqref{prope3}. If $(\mu, \nu)\in \mm(G+\{a_1, a_{2i+1}\})\times \mm(G+ \overline{\{a_1, a_{2i+1}\}})$ with $i>1$, we have the same construction with 
 $a_3$ replaced by $a_{2i+1}$.
 
 The map $(\mu, \nu)\mapsto (\mu^\pp, \nu^\pp)$ is invertible because given an element in $(\mu^\pp, \nu^\pp)$ of \eqref{prope3}, the pair $(\mu, \nu)$ that is mapped to it is obtained by shifting
  along the path in $\mu^\pp \cup \nu^\pp$ that contains the vertex $a_{2i}$, such that $(\mu^\pp, \nu^\pp)\in \mm(G+\{a_1, a_{2i}\})\times \mm(G+\overline{\{a_1, a_{2i}\}})$. The map we 
  have defined is weight-preserving and this proves the proposition.
\end{proof}

Now we can prove Theorem \ref{condensation-2}, which is essentially the same proof as that of Theorem \ref{condensation}, but now uses our more general Proposition \ref{ck3}.

\begin{proof}[Proof of Theorem \ref{condensation-2}]
 We prove the statement by induction on $k$. For $k=1$ it follows from the fact that \[ \pf \left( \begin{array}{cc}
0 & a \\
-a & 0 \end{array} \right) =a.\]

For the induction step, we assume that the statement holds for $k-1$ with $k\geq 2$. Let $A$ be the matrix 

\[\left( \begin{array}{ccccc}
0 & \m(G+\{a_1, a_2\}) & \m(G+\{a_1, a_3\}) & \cdots & \m(G+\{a_1, a_{2k}\}) \\
-\m(G+\{a_1, a_2\}) & 0 & \m(G+\{a_2, a_3\}) & \cdots & \m(G+\{a_2, a_{2k}\})\\
-\m(G+\{a_1, a_3\}) & -\m(G+\{a_2, a_3\}) & 0 & \cdots & \m(G+\{a_3, a_{2k}\}) \\
\vdots & \vdots & \vdots & & \vdots \\
-\m(G+\{a_1, a_{2k}\}) & -\m(G+\{a_2, a_{2k}\}) & -\m(G+\{a_3, a_{2k}\}) & \cdots & 0\end{array} \right).\]

\noindent By a well-known property of Pfaffians, we have
\begin{equation}\label{cond1}
\pf(A)=\sum_{i=2}^{2k}(-1)^i\m(G+\{a_1, a_i\})\pf (A_{1i}).
\end{equation}

Now, the induction hypothesis applied to the graph $G$ and the $2k-2$ vertices in $\overline{\{a_i, a_j\}}$ gives us 
\begin{equation}\label{cond2}
 [\m(G)]^{k-2}\m(G+\overline{\{a_1, a_i\}})=\pf(A_{1i}),
\end{equation}

\noindent where $A_{1i}$ is same as in equation \eqref{cond1}. So using equations \eqref{cond1} and \eqref{cond2} we get
\begin{equation}\label{cond3}
 \pf(A)=[\m(g)]^{k-2}\sum_{i=2}{2k}(-1)^i\m(G+\{a_1, a_i\})\m(G+\overline{\{a_1, a_i\}}).
\end{equation}

\noindent Now using Propositition \ref{ck3}, we see that the above sum is $\m(G)\m(G+\{a_1, \ldots, a_{2k}\})$ and hence equation \eqref{cond3} implies \eqref{cond4}.

\end{proof}

\section{Some family of regions with defects}\label{s3}

In this section, we find the number of tilings by dominoes of certain regions which appear in the statement of Theorem \ref{mt1} and Theorem \ref{mt3}. We define the binomial coefficients that appear in this section as follows

\begin{equation*}
\binom{c}{d} := \begin{cases} \dfrac{c(c-1)\cdots(c-d+1)}{d!}, &\text{if } d\geq0\\
0, &\text{otherwise} \end{cases}.
\end{equation*}

\noindent Our formulas also involve hypergeometric series. We recall that the hypergeometric series of parameters $a_1, \ldots, a_r$ and $b_1, \ldots, b_s$ is defined as 

\[_rF_s\left[\hyper{a_1, \ldots, a_r}{b_1, \ldots, b_s}\,;z\right]=\sum_{k=0}^{\infty}\frac{(a_1)_k\cdots (a_r)_k}{(b_1)_k\cdots (b_s)_k}\frac{z^k}{k!}.\]

We also fix a notation for the remainder of this paper as follows, if we remove the squares labelled $2,4,7$ from the south-eastern boundary of $\ar_{4,7}$, we denote it by $\ar_{4,7}(2,4,7)$. In the derivation of the results in this section, the following two corollaries of Theorem \ref{ar} will be used.

\begin{corollary}\label{cor1}
 The number of tilings of $\ar_{a,a+1}(i)$ is given by \[2^{a(a+1)/2}\binom{a}{i-1}.\]
\end{corollary}

\begin{corollary}\label{cor2}
 The number of tilings of $\ar_{a,b}(2,\ldots, b-a+1)$ is given by \[2^{a(a+1)/2} \binom{b-1}{a-1}.\]
\end{corollary}

\begin{figure}[!htb]
\centering
\includegraphics[scale=.7]{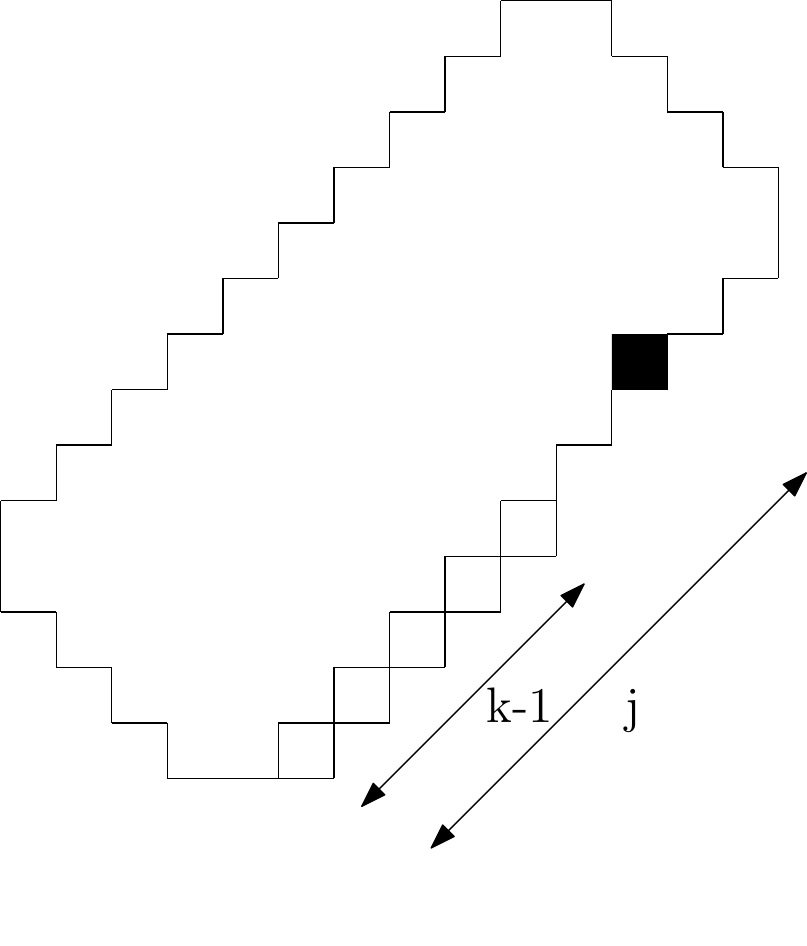}
\caption{Aztec rectangle with $k-1$ squares added on the southeastern side and a defect on the $j$-th position shaded in black; here $a=4,b=10,k=6,j=8$}
\label{fig:ar_k_i}
\end{figure}

\begin{proposition}\label{ar_k_i}
 
 Let $1\leq a\leq b$ be positive integers with $k=b-a>0$, then the number of domino tilings of $\ar_{a,b}(j)$ with $k-1$ squares added to the southeastern side 
 starting at the second position (and not at the bottom) as shown in the Figure \ref{fig:ar_k_i} is given by 
 
 \begin{equation}\label{k-1,j}
 2^{a(a+1)/2}\binom{a+k-1}{j-1}\binom{j-2}{k-1}~_3F_2\left[\hyper{1,1-j,1-k}{2-j, 1-a-k}\,;1\right].  
 \end{equation}
 
\end{proposition}

\begin{proof}
Let us denote the region in Figure \ref{fig:ar_k_i} by $\dk$ and we work with the planar dual graph of the region $\dk$ and count the number of matchings of 
that graph. We first notice that the first added square in any tiling of the region in Figure \ref{fig:ar_k_i} by dominoes has two possibilities marked in grey in the Figure 
\ref{fig:ar_k_i1}. This observation 
allows us to write the number of tilings of $\dk$ in terms of the following recursion

\begin{equation}\label{ep31}
 \m (\dk)=\m (\dkk)+\m (\ar_{a,b}(2,3,\ldots, k, j)).
\end{equation}

\noindent which can be verified from Figure \ref{fig:grey}.

\begin{figure}[!htb]
\centering
\includegraphics[scale=.7]{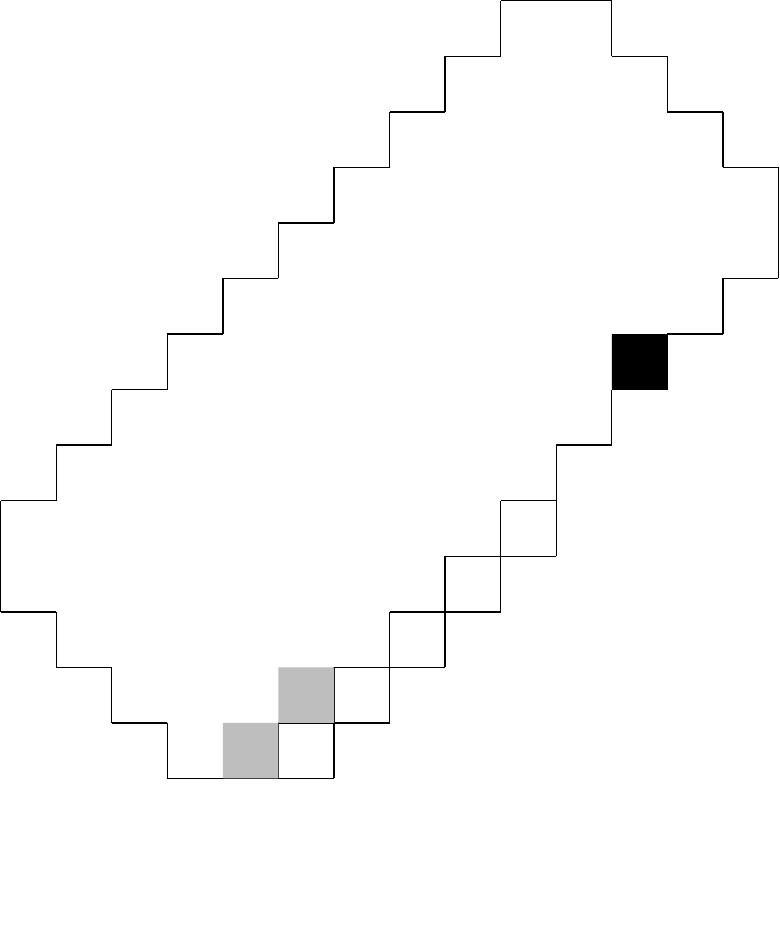}
\caption{$\dk$ with the possible choices for the first added square in a tiling; here $a=4, b=10, k=6, j=8$}
\label{fig:ar_k_i1}
\end{figure}

\noindent Repeatedly using equation \eqref{ep31} $k-1$ times on succesive iterations, we shall finally obtain
\begin{equation}\label{ep31-2}
\m (\dk)=\sum_{l=0}^{k-2}\m (\ar_{a,b-l}(2,3,\ldots,k-l,j-l))+\m (\ar_{a,a+1}(j-k+1)).
\end{equation}

\begin{figure}[!htb]
\minipage{0.50\textwidth}
  \includegraphics[width=\linewidth, center]{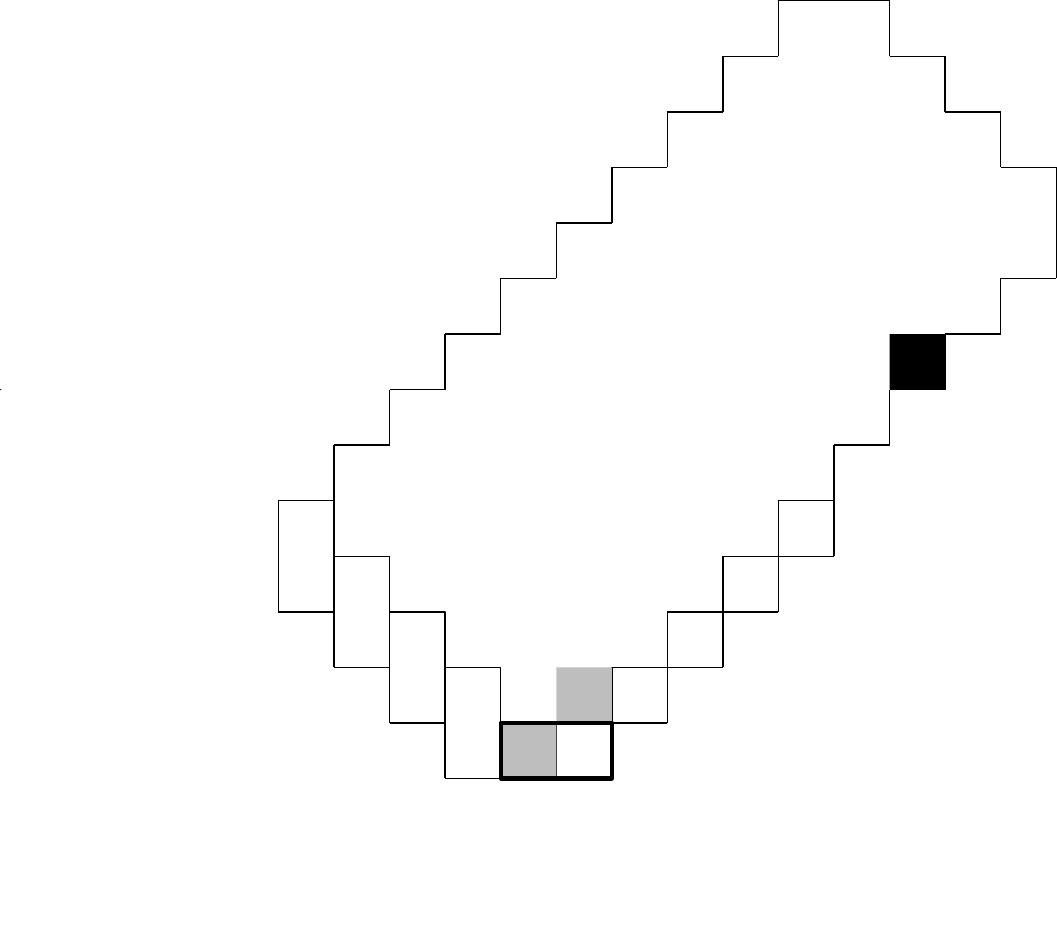}
\endminipage\hfill
\minipage{0.50\textwidth}
  \includegraphics[width=\linewidth, center]{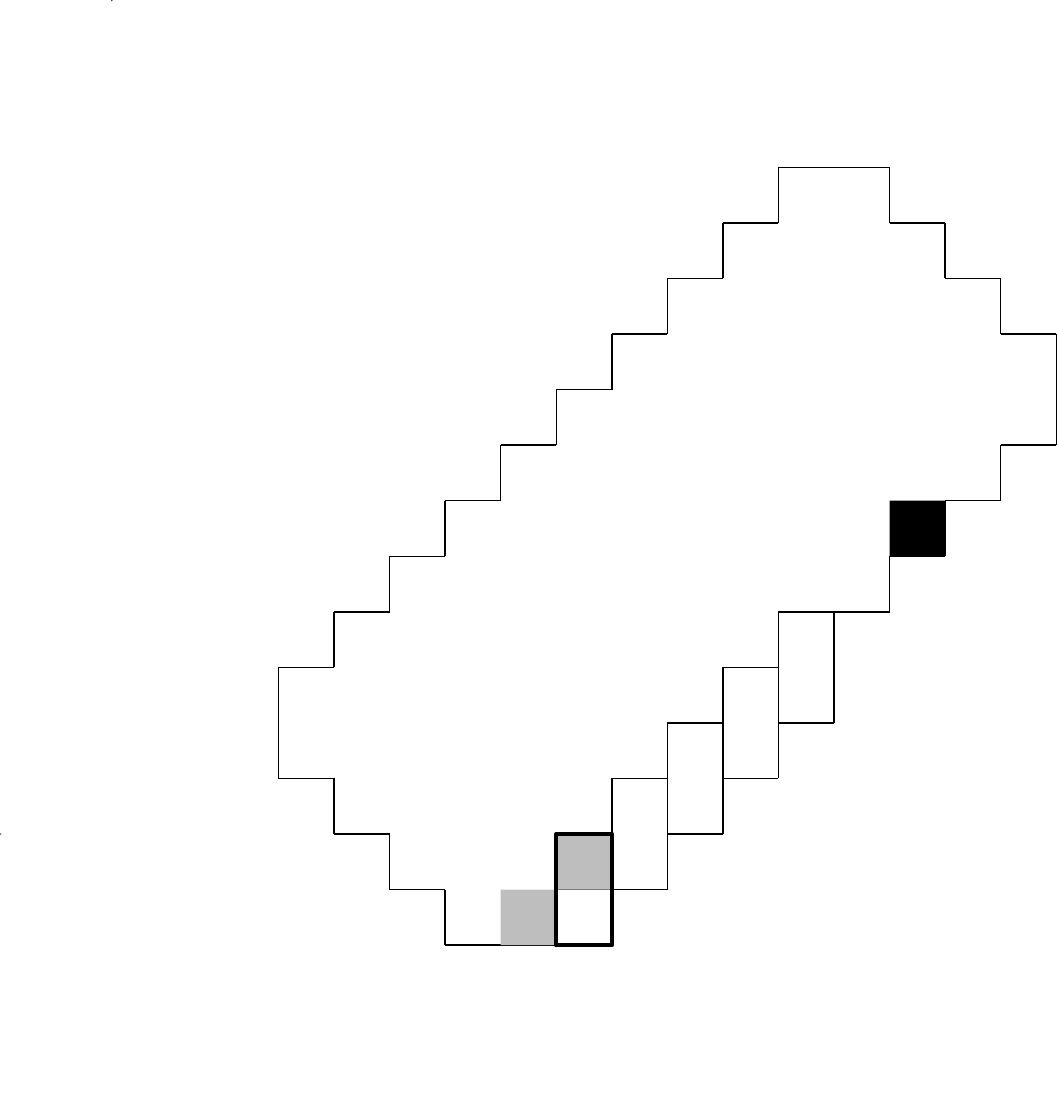}
\endminipage
\caption{Choices for the tilings of $\dk$ with forced dominoes; here $a=4,b=10,k=6,j=8$}
\label{fig:grey}
\end{figure}

Now, plugging in the values of the quantities in the right hand side of equation \eqref{ep31-2} from Theorem \ref{ar} and Corollary \ref{cor1} we shall obtain equation \eqref{k-1,j}.
 
\end{proof}

One of the main ingredients in our proofs of the remaining results in this section are the following results of Kuo \cite{kuo}.

\begin{theorem}\cite[Theorem 2.3]{kuo}\label{kk1}
Let $G=(V_1, V_2, E)$ be a plane bipartite graph in which $\abs{V_1}=\abs{V_2}$. Let $w, x, y$ and $z$ be vertices of $G$ that appear in cyclic order on a face of $G$. If $w,x\in V_1$ and $y, z\in V_2$ then 
$$\m (G-\{w,z\})\m (G-\{x,y\})=\m (G)\m (G-\{w,x,y,z\})+\m (G-\{w,y\})\m (G-\{x,z\}).$$
\end{theorem}

\begin{theorem}\cite[Theorem 2.5]{kuo}\label{kk}
Let $G=(V_1, V_2, E)$ be a plane bipartite graph in which $\abs{V_1}=\abs{V_2}+2$. Let the vertices $w,x,y$ and $z$ appear in that cyclic order on a face of $G$. Let $w,x,y,z\in V_1$, then 
$$\m (G-\{w,y\})\m (G-\{x,z\})=\m (G-\{w,x\})\m (G-\{y,z\})+\m (G-\{w,z\})\m (G-\{x,y\}).$$
\end{theorem}

The following proposition does not appear explicitely in the statement of Theorem \ref{mt1}, but it is used in deriving Proposition \ref{ar_k-1_i}.

\begin{figure}[!htb]
\centering
\includegraphics[scale=.7]{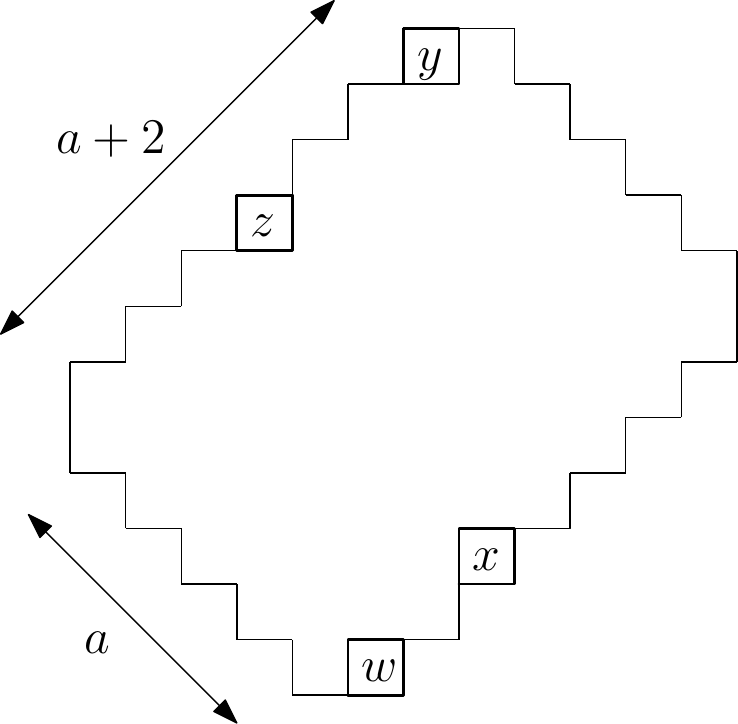}
\caption{An $a\times (a+2)$ Aztec rectangle with some labelled squares; here $a=5$}
\label{fig:ar_i_j}
\end{figure}

\begin{proposition}\label{ar_i_j}
 
 Let $1\leq a$ be a positive integer, then the number of tilings of $\ar_{a, a+2}$ with a defect at the $i$-th position on the southeastern side counted from the south corner and a defect on the $j$-th position on the northwestern side 
 counted from the west corner is given by 
 \begin{equation}\label{arij}
2^{a(a+1)/2}\left[\binom{a}{i-2}\binom{a}{j-1}+\binom{a}{i-1}\binom{a}{j-2}\right].
\end{equation}

\end{proposition}

\begin{proof}

\begin{figure}[!htb]
\minipage{0.50\textwidth}
  \includegraphics[scale=0.6, center]{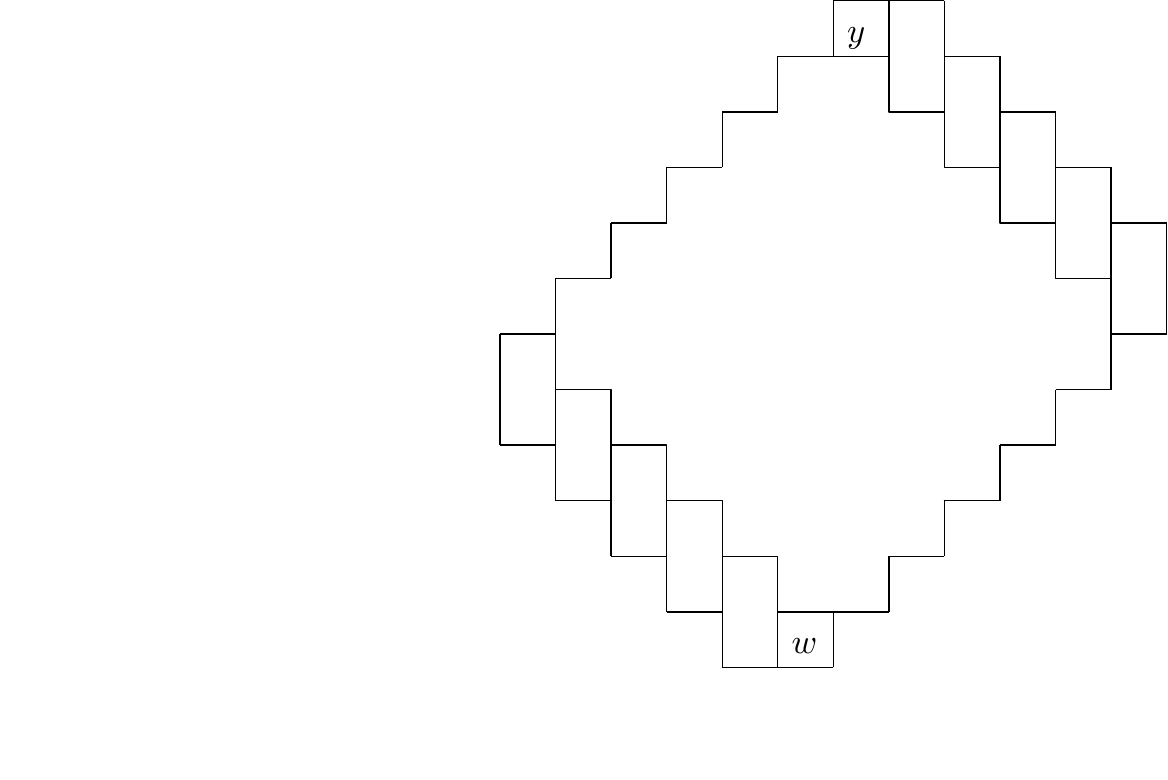}
\endminipage\hfill
\minipage{0.50\textwidth}
  \includegraphics[scale=0.6, center]{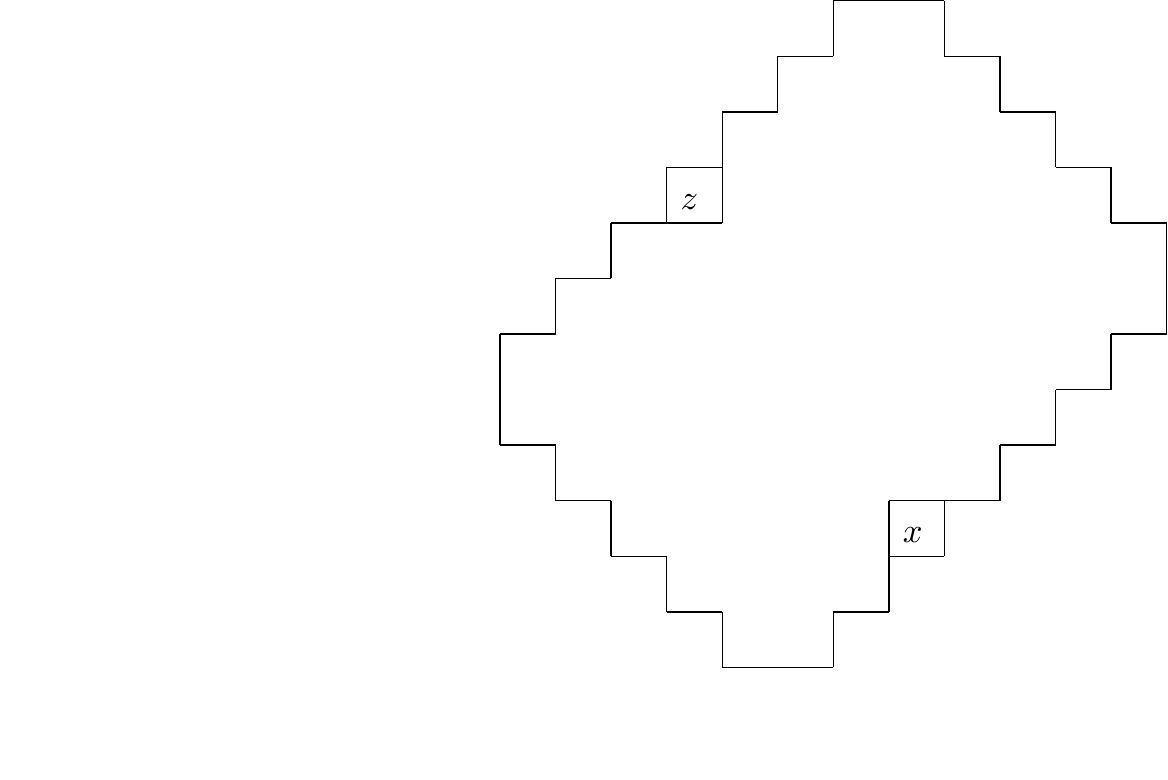}
  \endminipage\hfill
\minipage{0.50\textwidth}
  \includegraphics[scale=0.6, center]{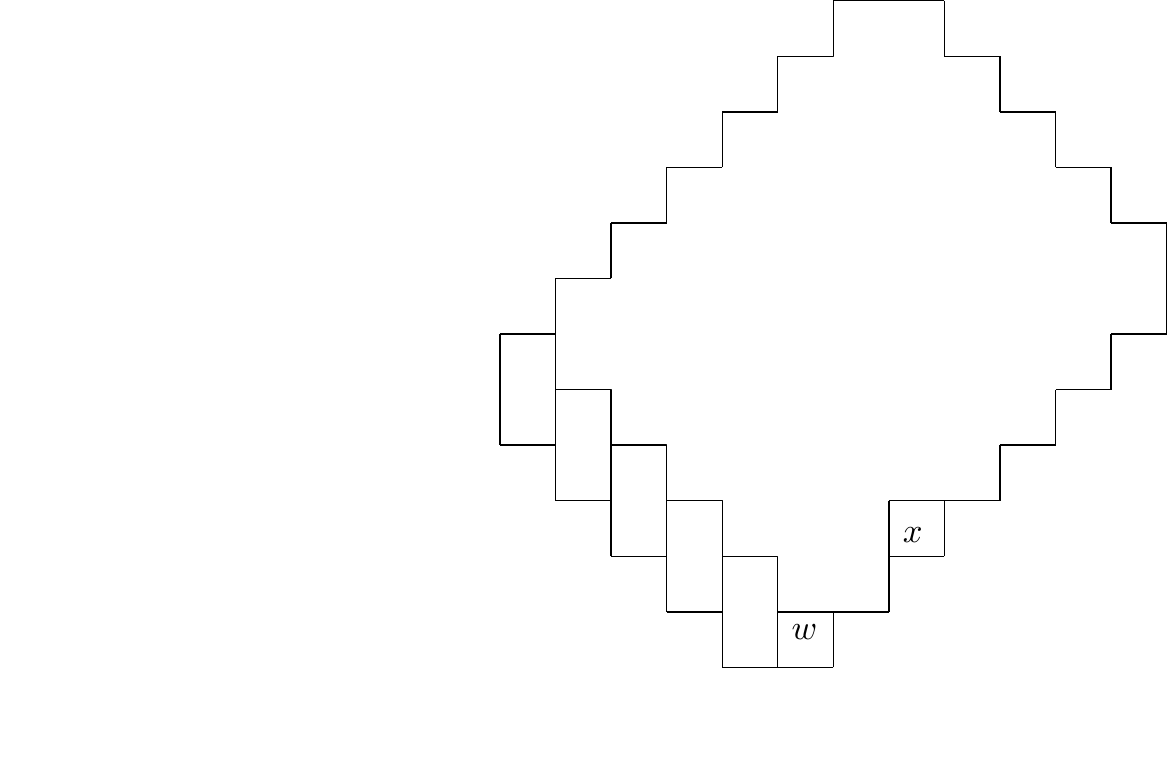}
  \endminipage\hfill
\minipage{0.50\textwidth}
  \includegraphics[scale=0.6, center]{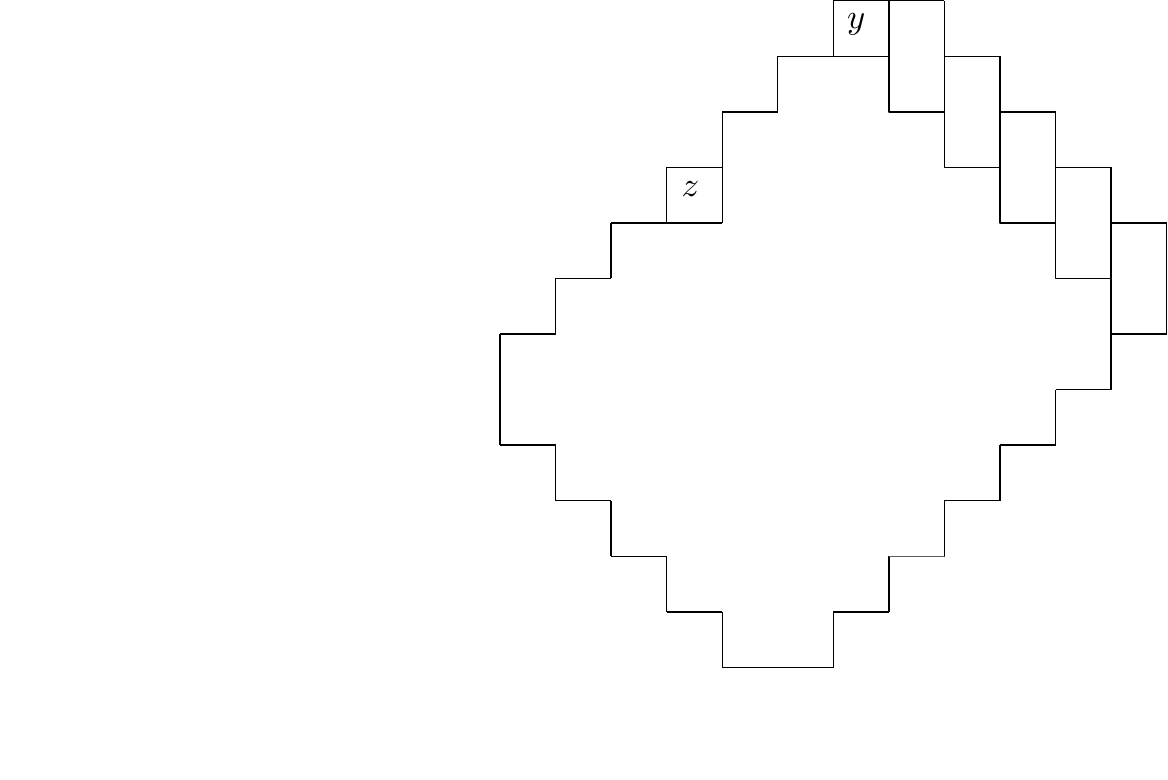}
  \endminipage\hfill
\minipage{0.50\textwidth}
  \includegraphics[scale=0.6, center]{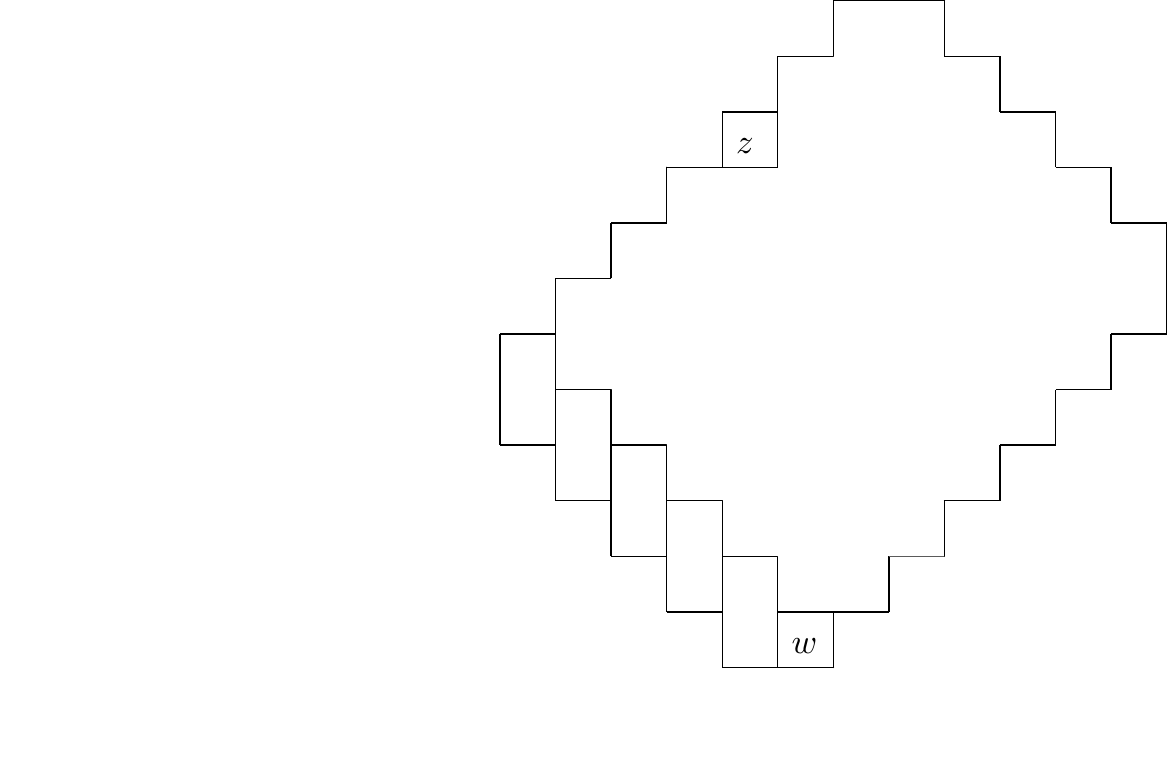}
  \endminipage\hfill
\minipage{0.50\textwidth}
  \includegraphics[scale=0.6, center]{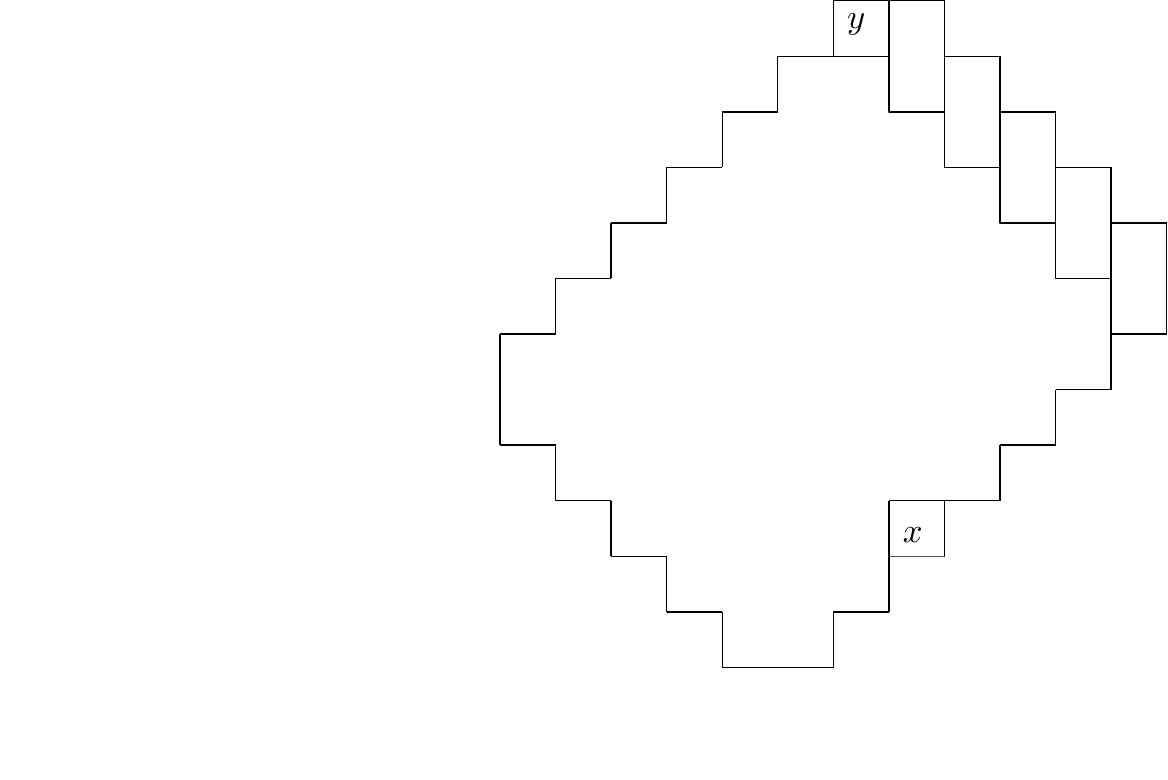}
\endminipage
\caption{Some forced dominoes in the proof of Proposition \ref{ar_i_j} where the vertices we remove are labelled}
\label{fig:kuo-1}
\end{figure}

If $j=1$ or $j=a+2$, then the region we want to tile reduces to the type in Theorem \ref{ar} and it is easy to see that the expression \eqref{arij} is satisfied in these cases. By symmetry, this 
also takes care of the cases $i=1$ and $i=a+2$.
 
In the rest of the proof, we now assume that $1<i,j<a+2$ and let us denote the region we are interested in by $\oo$. We now use Theorem \ref{kk} with the vertices as indicated in Figure \ref{fig:ar_i_j} to obtain the following identity
 (Figure \ref{fig:kuo-1}).
 
 \begin{align}\label{ep32}
  \m (\ad(a))\m (\oo) =& \m (\ar_{a,a+1}(i-1))\m (\ar_{a,a+1}(j))\\ \nonumber
   &+ \m (\ar_{a,a+1}(j-1))\m (\ar_{a,a+1}(i)).
 \end{align}

\noindent Now, using Theorem \ref{adm} and Corollary \ref{cor1} in equation \eqref{ep32} we get \eqref{arij}.
 
\end{proof}

\begin{remark}\label{rem1}
 
Ciucu and Fischer \cite{ilse}, have a similar result for the number of lozenge tiling of a hexagon with dents on opposite sides (Proposition 4 in their paper). They also make use of 
Kuo's condensation result, Theorem \ref{kk1} and obtain the following identity

\begin{align*}
\opp(a,b,c)_{i,j}&\opp(a-2,b,c)_{i-1, j-1} \\
=& \opp(a-1,b,c)_{i-1,j-1}\opp(a-1,b,c)_{i,j}\\
 &- \opp(a-1,b-1,c+1)_{i,j-1}\opp(a-1,b+1,c-1)_{i-1,j} 
\end{align*}

\noindent where $\opp(a,b,c)_{i,j}$ denotes the number of lozenge tilings of a hexagon $H_{a,b,c}$ with opposite side lengths $a,b,c$ and with two dents in position $i$ and $j$ on 
opposite sides of length $a$, where $a,b,c,i,j$ are positive integers with $1\leq i,j \leq a$. 

In their use of Kuo's result, they take the graph $G$ to be $\opp(a,b,c)_{i,j}$, but 
if we take the graph $G$ to be $H_{a,b,c}$ and use Theorem \ref{kk1} with an appropriate choice of labels, we get the following identity

\begin{align*}
\opp(a,b,c)_{i,j}\hex(a-1,b,c) =& \hex(a,b,c)\opp(a-1,b,c)_{i,j} \\ 
 &+ \hex(a,c-1,b+1,a-1,c,b)_i\hex(a,c,b,a-1,c+1,b-1)_{a-j+1}
\end{align*}

\noindent where $\hex(a,b,c)$ denotes the number of lozenge tilings of the hexagon with opposite sides of length $a,b,c$ and $\hex(m,n,o,p,q,r)_k$ denotes the number of 
lozenge tilings of a hexagon with side lengths $m,n,o,p,q,r$ with a dent at position $k$ on the side of length $m$. Then, Proposition 4 of Ciucu and Fischer \cite{ilse} follows more easily without the need for contigous relations of hypergeometric series that they use in their paper.
 
\end{remark}

\begin{figure}[!htb]
\centering
\includegraphics[scale=.7]{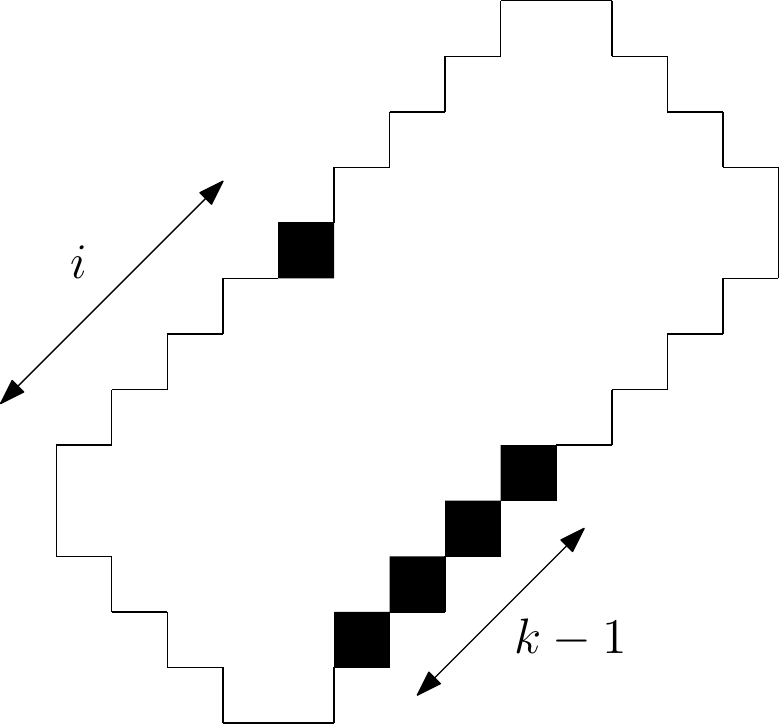}
\caption{An $a\times b$ Aztec rectangle with defects marked in black; here $a=4, b=9. k=5, i=5$}
\label{fig:ar_k-1_i}
\end{figure}

 \begin{proposition}\label{ar_k-1_i}
Let $1\leq a, i \leq b$ be positive integers with $k=b-a>0$, then the number of domino tilings of $\ar_{a,b}(2, 3, \ldots, k)$ with a defect on the northwestern side in the $i$-th position counted from the west corner as shown in the Figure \ref{fig:ar_k-1_i} is given by

 \[2^{a(a+1)/2}\binom{a+k-2}{k-1}\binom{a}{a-i+k}~_3F_2\left[\hyper{1, -k-1, i-a-k}{i-k+1, 2-a-k}\,;-1\right].\]

\end{proposition}

\begin{proof}
 
Our proof will be by induction on $b=a+k$. The base case of induction will follow if we verify the 
result for $a=2, k=1$ in which case $b=3$. We also need to check the result for $i=1$ and $i=b$. If $i=1$ we have many forced dominoes and we get the region shown in Figure \ref{fig:prop34-2}, which is $\ad(a)$. Again, if $i=b$, then also we get a region of the type in Theorem \ref{ar}.
 In both of these cases the number of domino tilings of these regions satisfy the formula mentioned in the statement. To check our base case it is now enough to verify the formula for $a=2, k=1, i=2$ as the other cases of $i=1$ 
 and $i=3$ are already taken care of. In this case, we see that the region we obtain is of the type as described in Corollary \ref{cor1} and this satisfies the statement of our result. 

\begin{figure}[!htb]
\centering
\includegraphics[scale=.7]{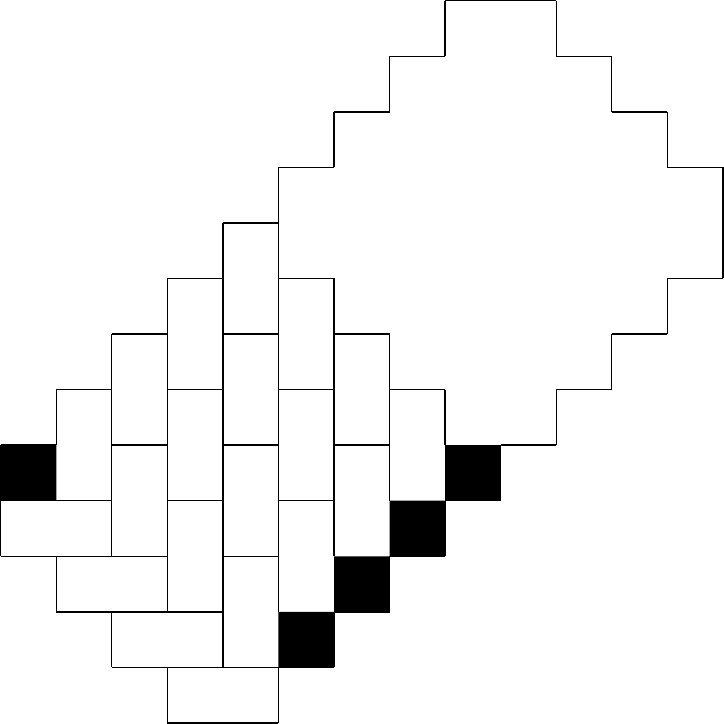}
\caption{Forced tilings for $i=1$ in Proposition \ref{ar_k-1_i}}
\label{fig:prop34-2}
\end{figure}

From now on, we assume $b> 3$ and $1<i<b$. We denote the region of the type shown in Figure \ref{fig:ar_k-1_i} by $\ar_{a,b,k-1}^{i}$. We use Theorem \ref{kk} here, with the vertices $w,x,y$ and $z$ 
marked as shown in Figure \ref{fig:prop34-1}, where we add a series of unit squares to the northeastern side to make it into an $a\times (b+1)$ Aztec rectangle. Note that the square in the 
$i$-th position to be removed is included in this region and is labelled by $z$. The identity we now obtain is the following (see Figure \ref{fig:kuo-2} for forcings)

\begin{figure}[!htb]
\centering
\includegraphics[scale=.7]{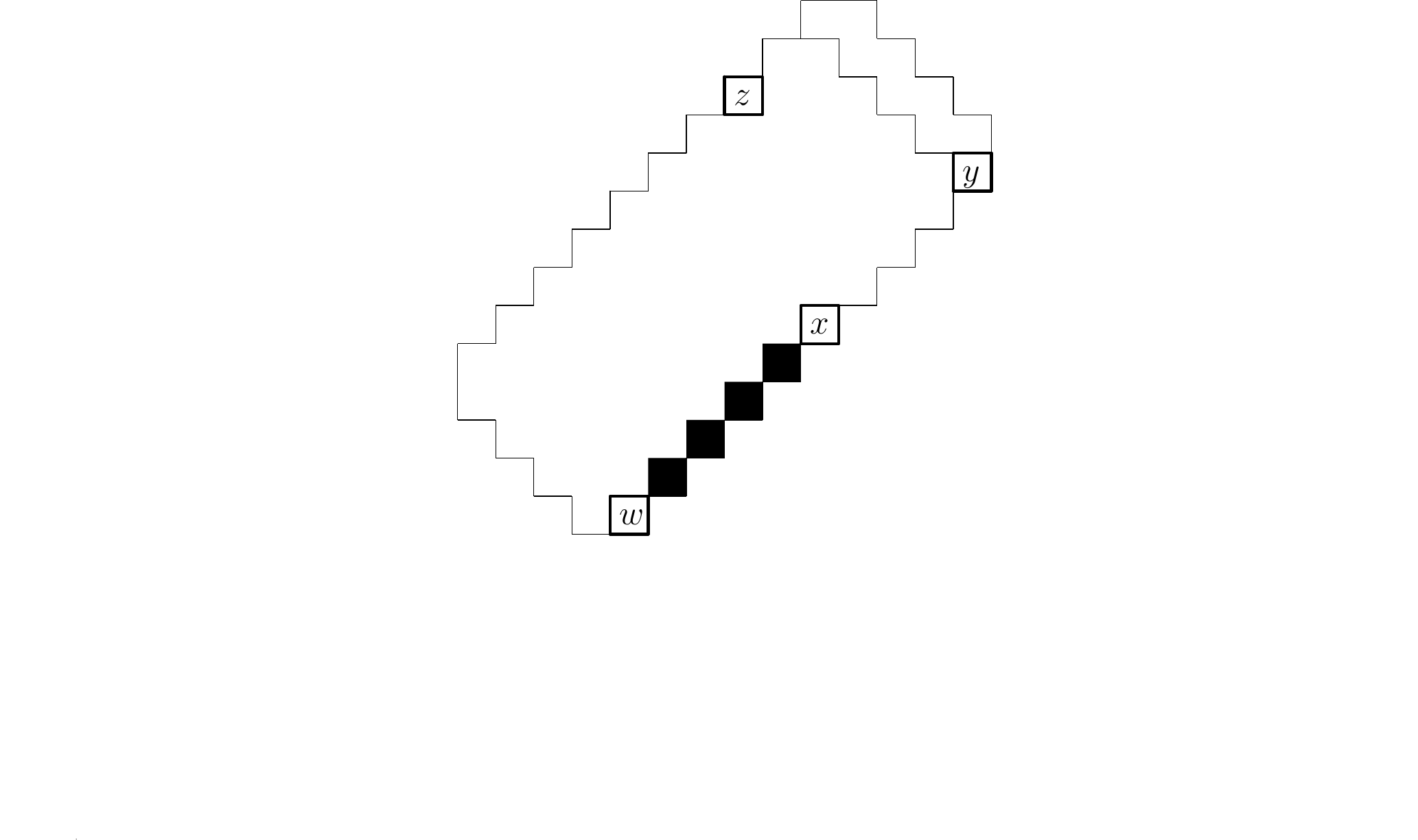}
\caption{Labelled $a\times (b+1)$ Aztec rectangle; here $a=4, b=9$}
\label{fig:prop34-1}
\end{figure}

\begin{equation}\label{ep34}
 \m (\ad(a))\m (\ar_{a,b+1,k}^{i}) = \m (\ad(a))\m (\ar_{a,b,k-1}^{i}) + Y\cdot \m (\ar_{a,b}(2,3,\ldots, k, k+1))
\end{equation}

 \noindent where

 \begin{equation}\label{ep34-1}
Y := \begin{cases} 0, &\text{if } i\leq k\\
\m (\ar_{a,a+1}(a+k+2-i), &\text{if } i\geq k+1\end{cases}.
\end{equation}

\begin{figure}[!htbp]
\minipage{0.50\textwidth}
  \includegraphics[scale=0.5, center]{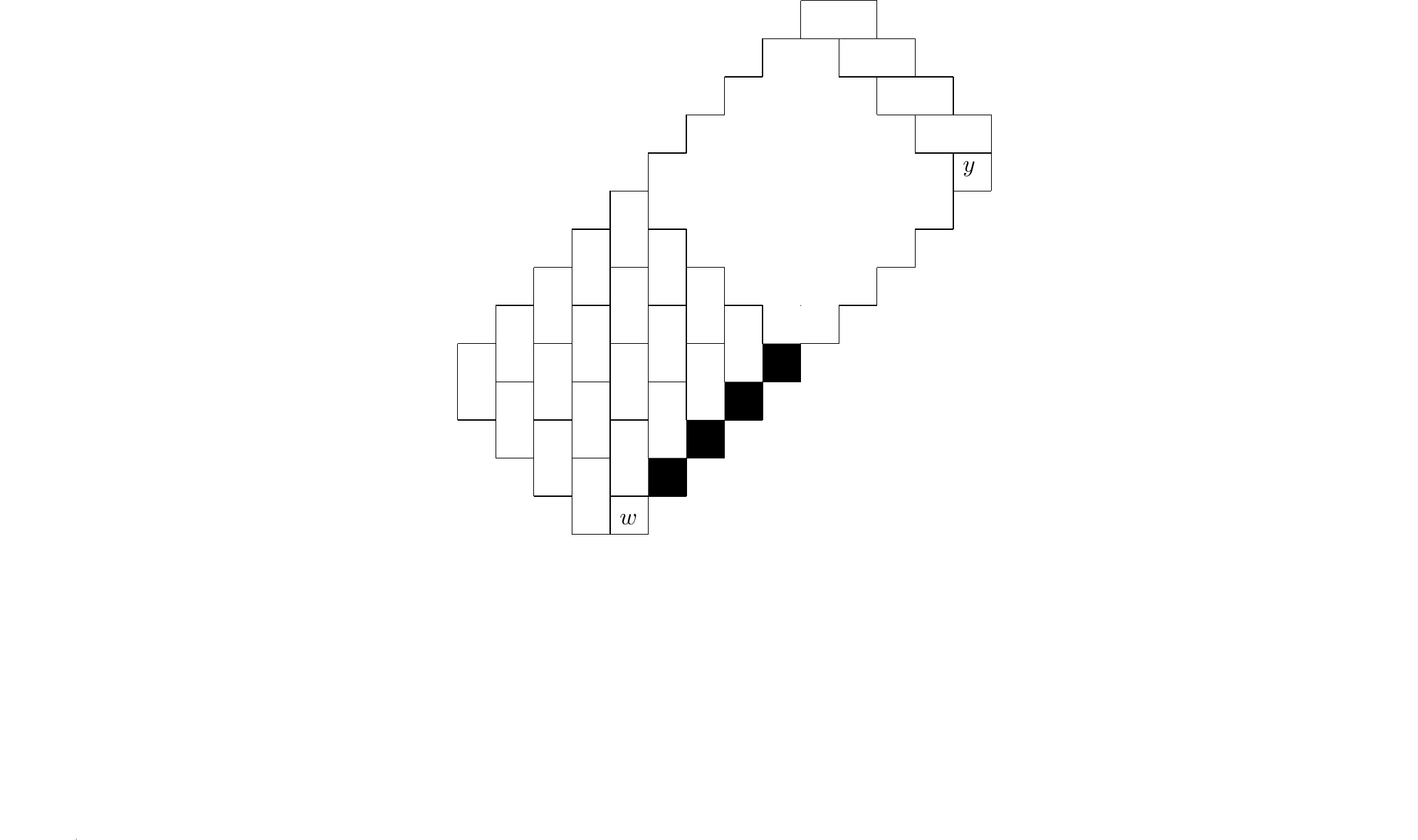}
\endminipage\hfill
\minipage{0.50\textwidth}
  \includegraphics[scale=0.5, center]{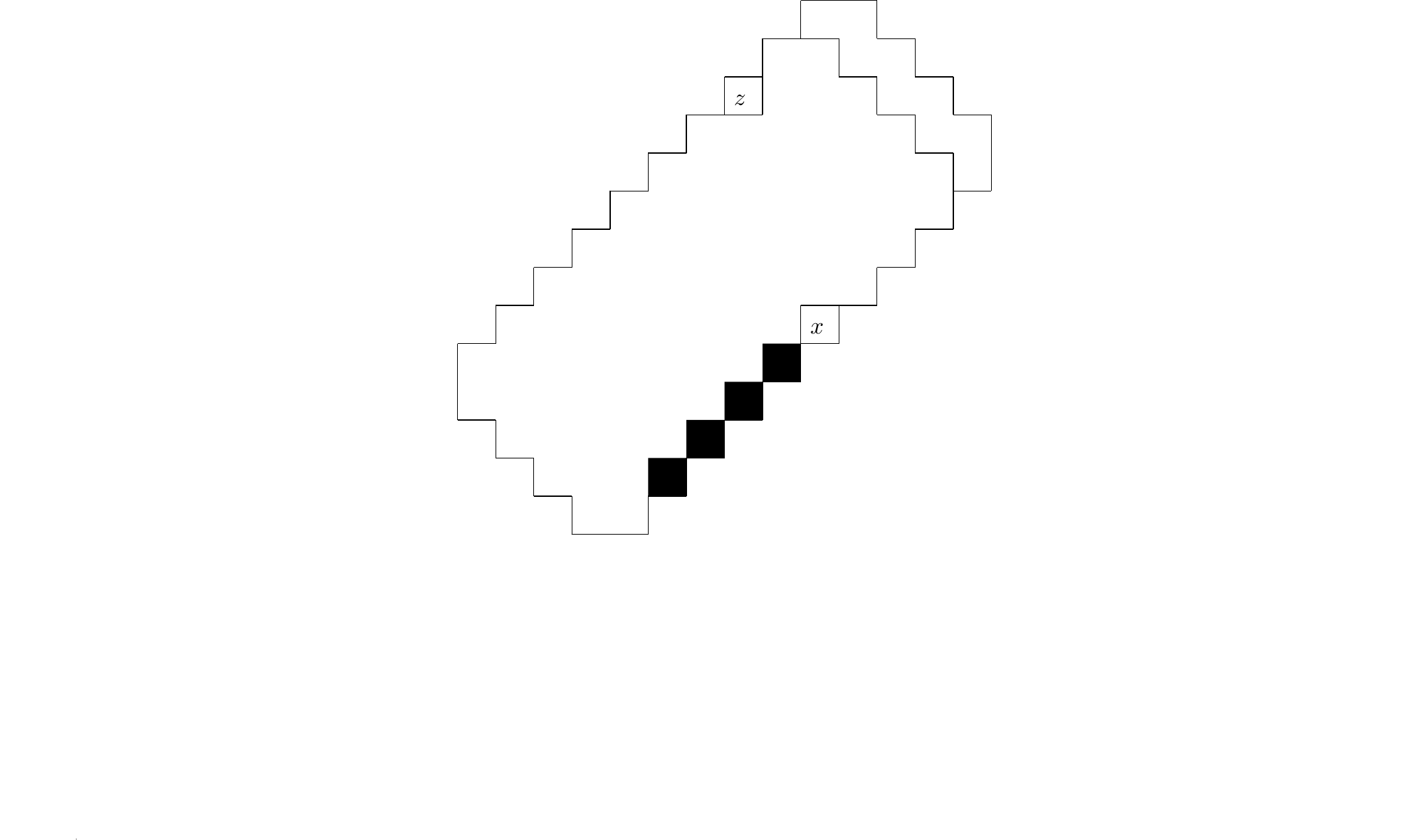}
  \endminipage\hfill
\minipage{0.50\textwidth}
  \includegraphics[scale=0.5, center]{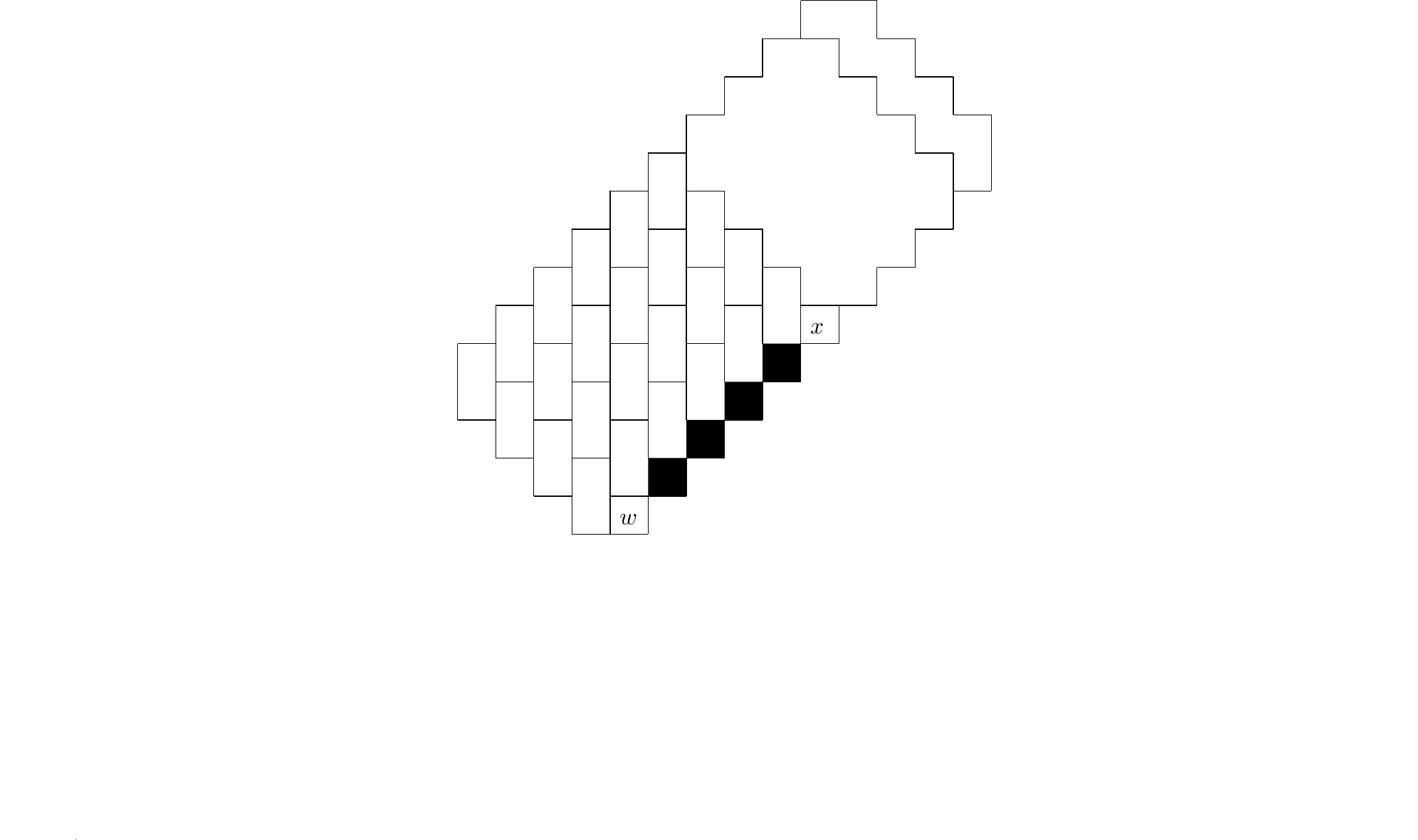}
  \endminipage\hfill
\minipage{0.50\textwidth}
  \includegraphics[scale=0.5, center]{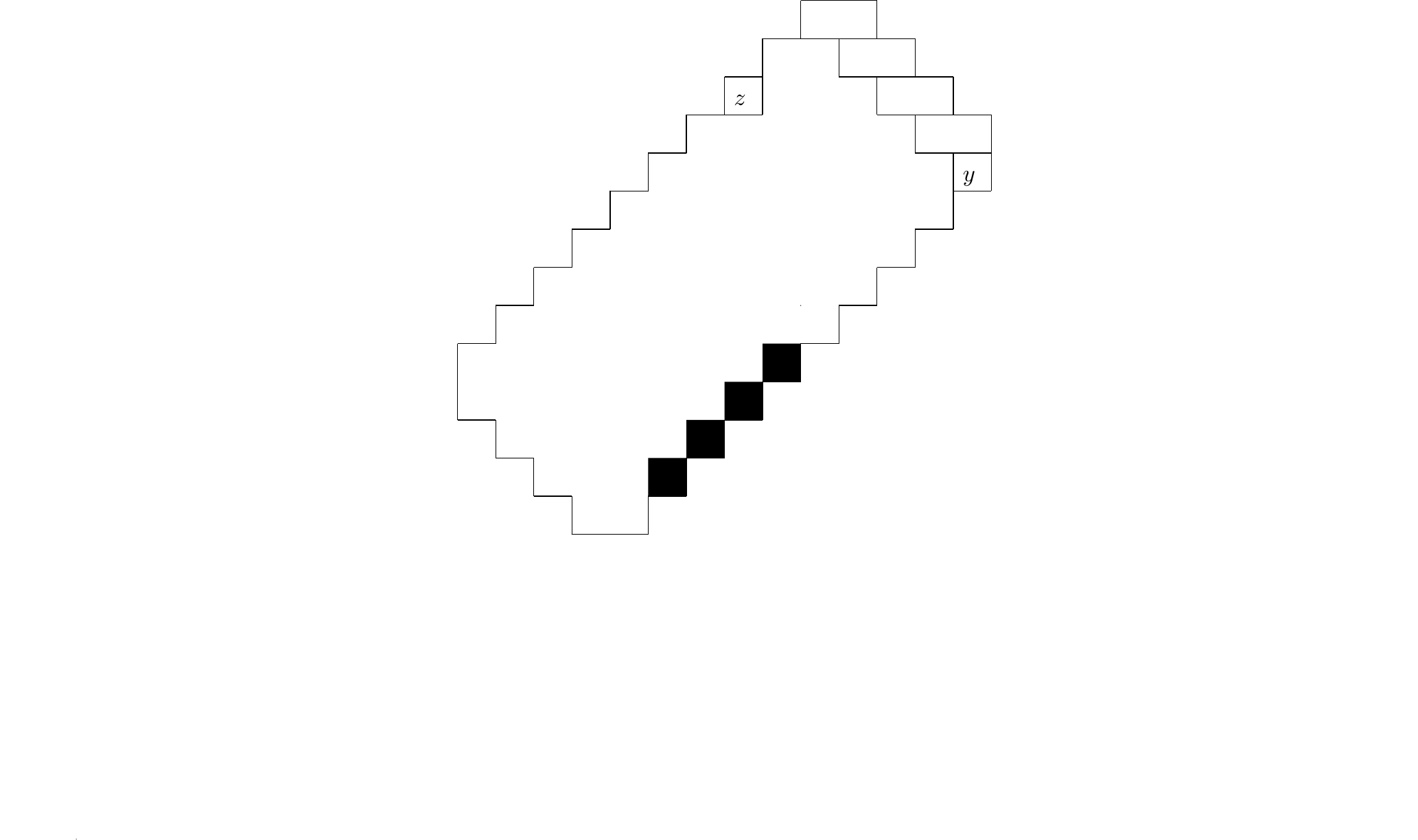}
  \endminipage\hfill
\minipage{0.50\textwidth}
  \includegraphics[scale=0.5, center]{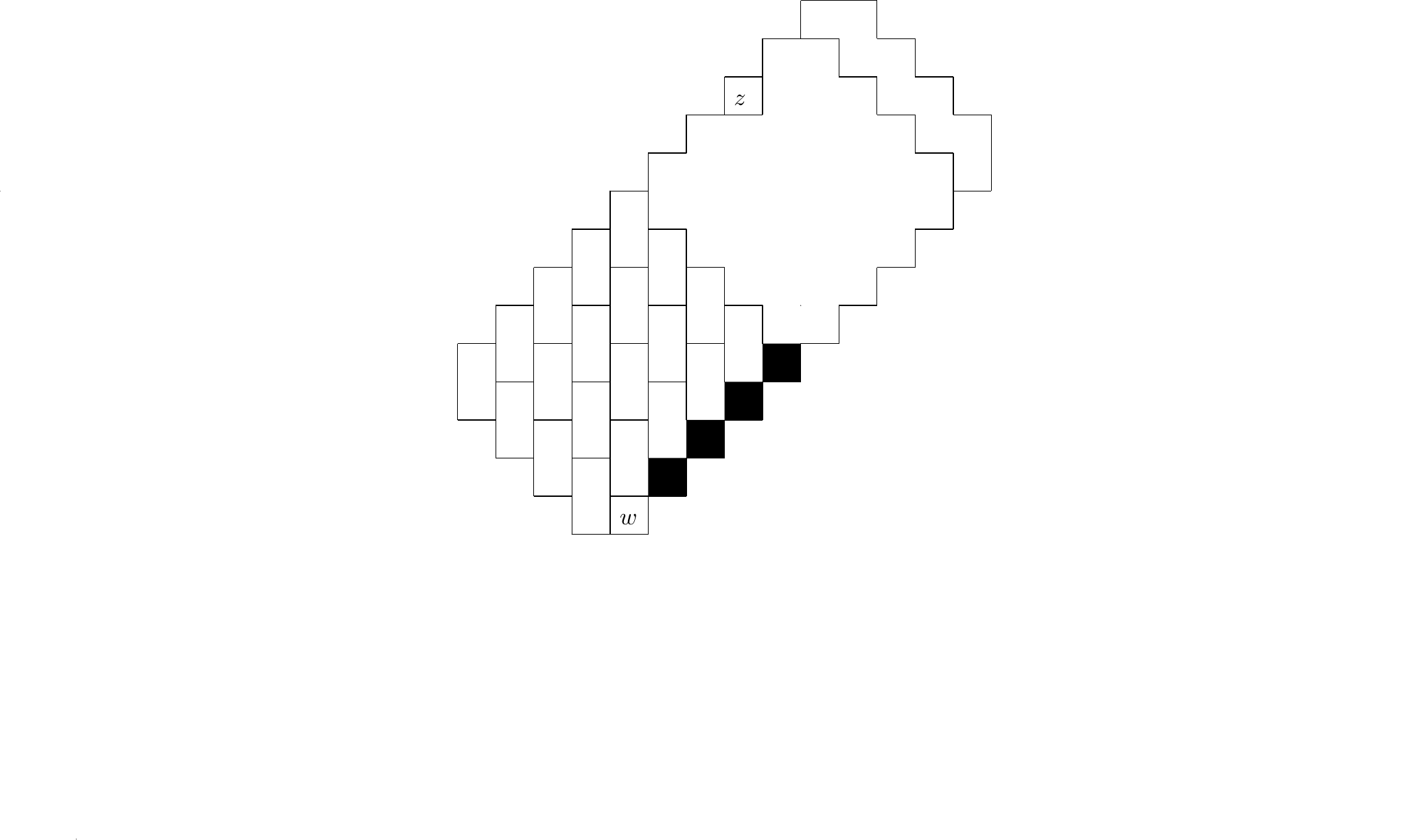}
  \endminipage\hfill
\minipage{0.50\textwidth}
  \includegraphics[scale=0.5, center]{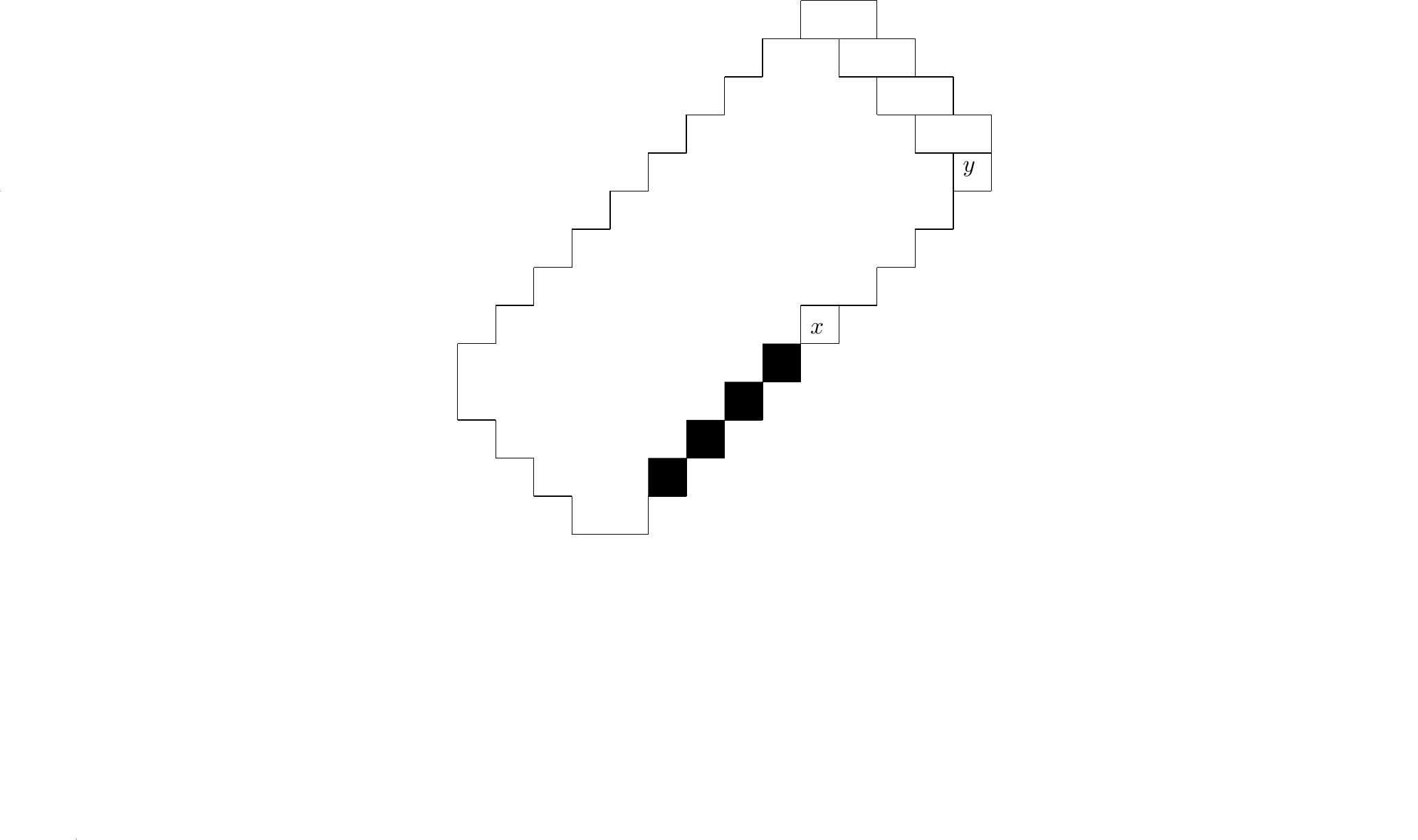}
\endminipage
\caption{Forced dominoes in the proof of Proposition \ref{ar_k-1_i} where the vertices we remove are labelled}
\label{fig:kuo-2}
\end{figure}

Using equation \eqref{ep34-1} in equation \eqref{ep34}, we can simplify the relation further to the following

\begin{equation}\label{ep34-m}
 \m (\ar_{a,b+1,k}^{i})=\m (\ar_{a,b,k-1}^{i})+Z\cdot \m (\ar_{a,b}(2,3,\ldots, k+1))
\end{equation}

\noindent where

 \begin{equation}\label{ep34-m-1}
Z := \begin{cases} 0, &\text{if } i\leq k\\
\dfrac{\m (\ar_{a,a+1}(a+k+2-i)}{\m (\ad(a))}, &\text{if } i\geq k+1\end{cases}.
\end{equation}

It now remains to show that the expression in the statement satisfies equation \eqref{ep34-1}. This is now a straightforward application of the induction hypothesis and some 
algebraic manipulation.
\end{proof}

\begin{figure}[!htb]
\centering
\includegraphics[scale=.7]{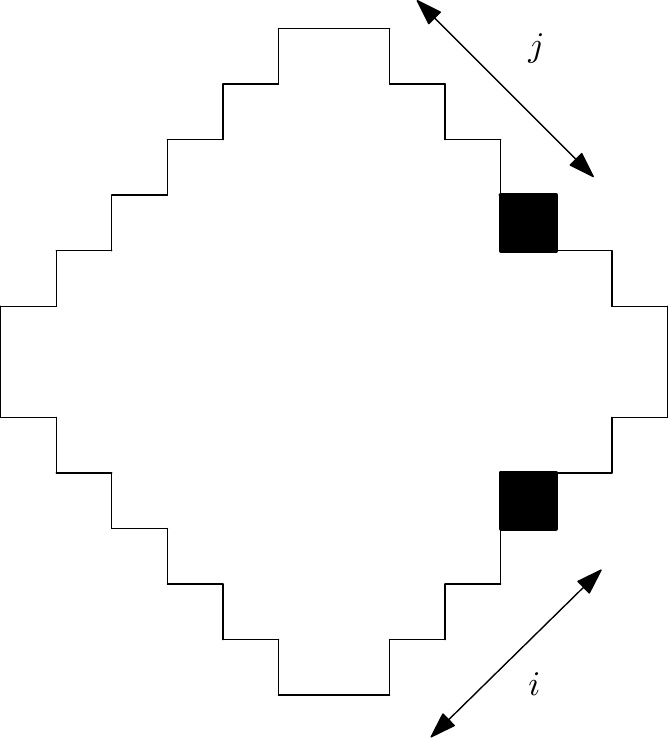}
\caption{Aztec Diamond with defects on adjacent sides; here $a=6$, $i=4$, $j=4$}
\label{fig:ad-i-j-n}
\end{figure}

 \begin{proposition}\label{ad_i_j}
  
 Let $a, i, j$ be positive integers such that $1\leq i, j\leq a$, then the number of domino tilings of $\ad(a)$ with one defect on the southeastern side at the $i$-th position counted from the south corner and one defect on the northeastern side on the $j$-th position counted from the north corner 
 as shown in Figure \ref{fig:ad-i-j-n} is given by 
 
\[ 2^{a(a-1)/2}\binom{a-1}{i-1}\binom{a-1}{j-1}~_3F_2\left[\hyper{1, 1-i, 1-j}{1-a, 1-a}\,;2\right].\]

 \end{proposition}

\begin{figure}[!htb]
\centering
\includegraphics[scale=.7]{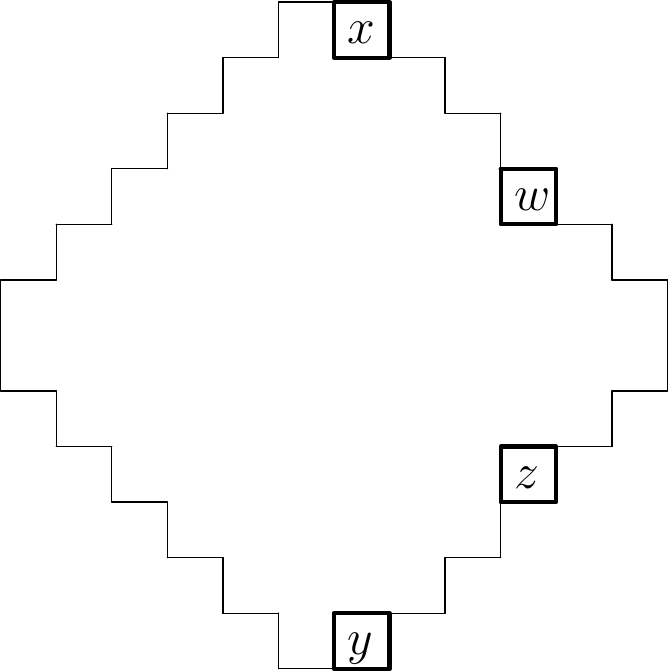}
\caption{Aztec Diamond with some labelled squares; here $a=6$}
\label{fig:ad_i_j}
\end{figure}

\begin{proof}

 We use induction with respect to $a$. The base case of induction is $a=2$. We would also need to check for $i=1,j=1,i=a$ and $j=a$ separately.

  If $a=2$, then the only possibilities are $i=1$ or $i=a$ and $j=1$ or $j=a$, so we do not have to consider this case, once we consider the other mentioned cases.

 We now note that when either $i$ or $j$ is $1$ or $a$, some dominoes are forced in any tiling 
 and hence we are reduced to an Aztec rectangle of size $(a-1)\times a$. It is easy to see that our formula is correct for this.

 In the rest of the proof we assume $a\geq 3$ and $1<i,j<a$. Let us now denote the region we are interested in this proposition as $\ad_a(i,j)$. Using the dual graph of this region and applying 
 Theorem \ref{kk1} with the vertices as labelled in Figure \ref{fig:ad_i_j} we obtain the following identity (see Figure \ref{fig:kuo-3} for details),

 \begin{align}\label{ep35}
    \m (\ad_a(i,j))\m (\ad(a-1)) =& \m (\ad(a))\m (\ad_{a-1}(i-1,j-1))\\ \nonumber
     &+ \m (\ar_{a-1,a}(j))\m (\ar_{a-1,a}(i)).
 \end{align}

\begin{figure}[!htb]
\minipage{0.50\textwidth}
  \includegraphics[scale=0.6, center]{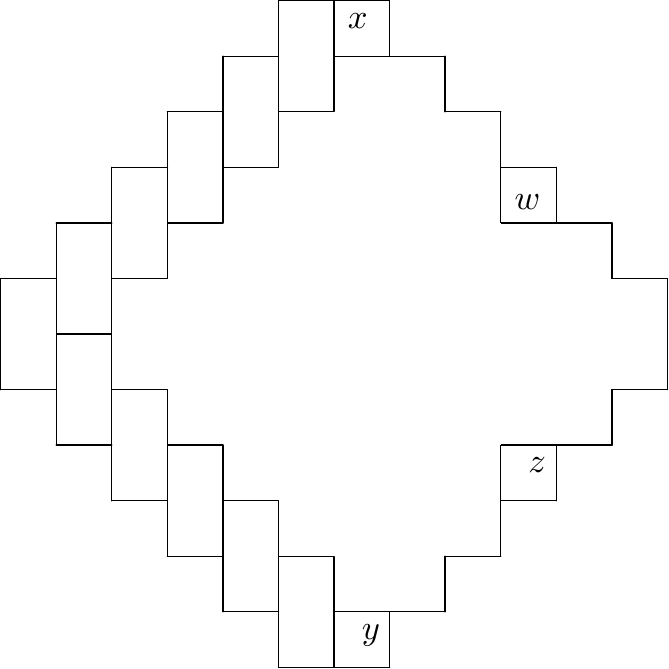}
  \endminipage\hfill
\minipage{0.50\textwidth}
  \includegraphics[scale=0.6, right]{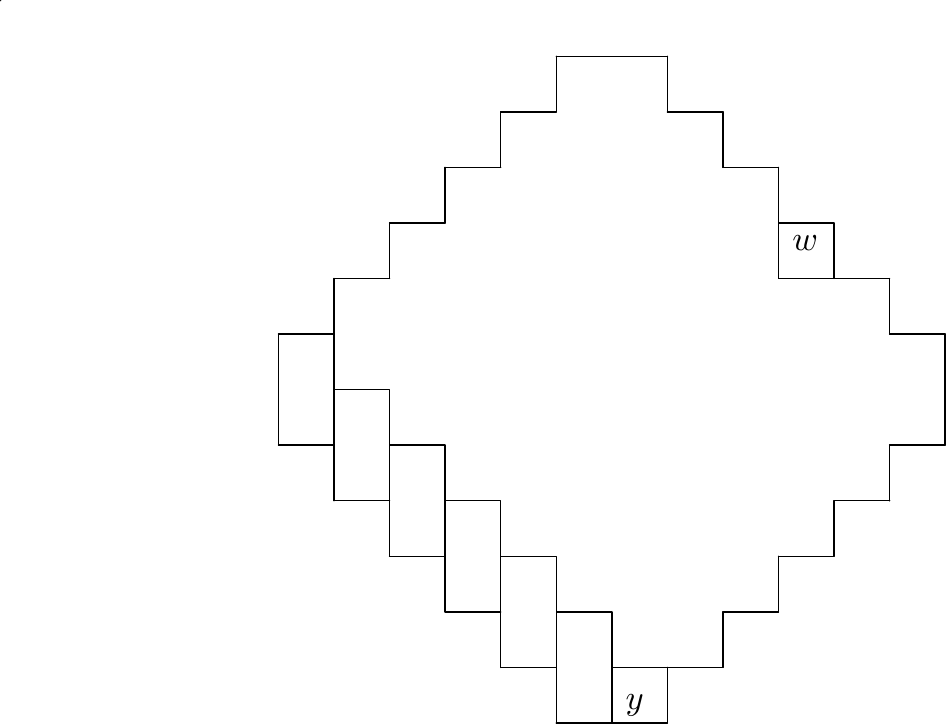}
  \endminipage\hfill
  \vspace{4mm}
\minipage{0.50\textwidth}
  \includegraphics[scale=0.6, left]{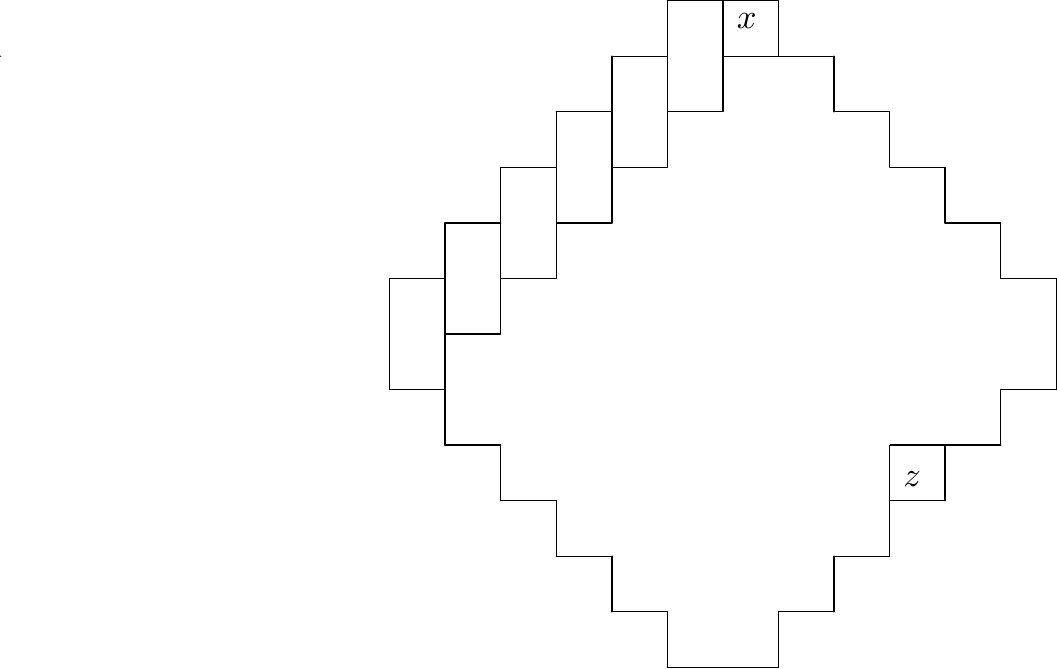}
  \endminipage\hfill
\minipage{0.50\textwidth}
  \includegraphics[scale=0.6, center]{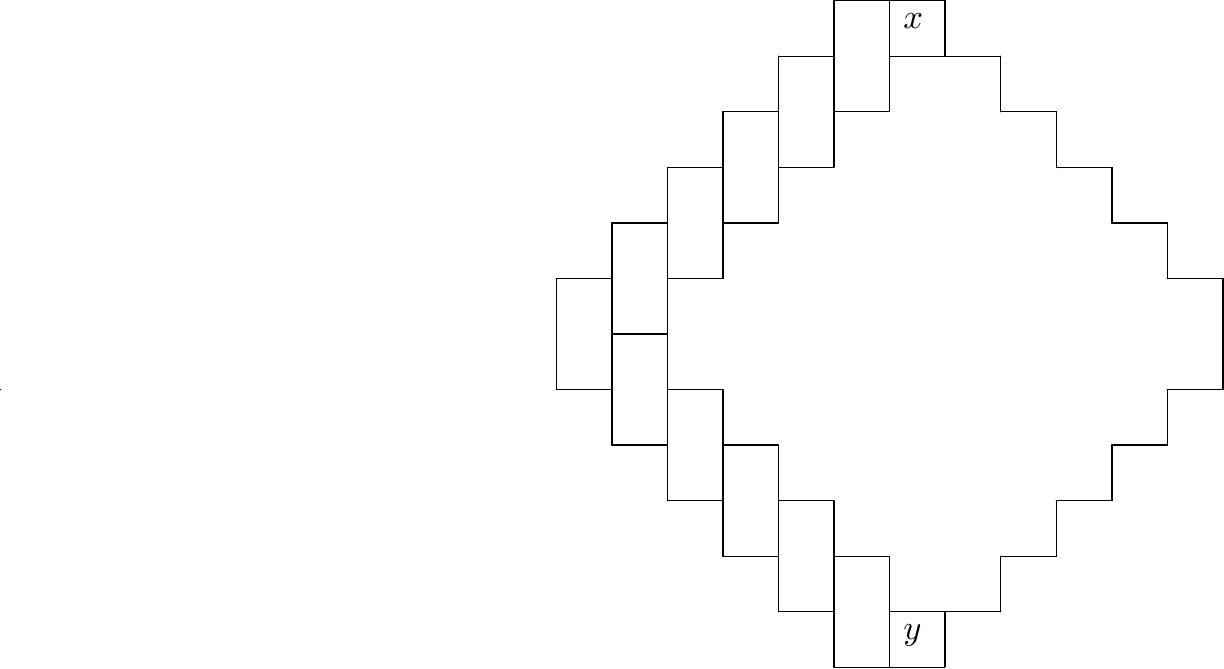}
\endminipage
\caption{Forced dominoes in the proof of Proposition \ref{ad_i_j} where the vertices we remove are labelled}
\label{fig:kuo-3}
\end{figure}

 \noindent Simplifying equation \eqref{ep35}, we get the following
 
 \begin{equation}\label{ll}
 \m (\ad_a(i,j))=2^a\m (\ad_{a-1}(i-1,j-1))+2^{a(a-1)/2}\binom{a-1}{j-1}\binom{a-1}{i-1}
 \end{equation}
 \noindent where we used Theorem \ref{adm} and Corollary \ref{cor1}.

 Now, using our inductive hypothesis on equation \eqref{ll} we see that we get the expression in the proposition.
\end{proof}

\begin{remark}
 Ciucu and Fischer \cite{ilse}, have a similar result for the number of lozenge tiling of a hexagon with dents on adjacent sides (Proposition 3 in their paper). They make use of the 
 following result of Kuo \cite{kuo}.
 
 \begin{theorem}[Theorem 2.1]\cite{kuo}\label{kj}
  Let $G=(V_1, V_2, E)$ be a plane bipartite graph with $\abs{V_1}=\abs{V_2}$ and $w,x,y,z$ be vertices of $G$ that appear in cyclic order on a face of $G$. If $w,y \in V_1$ 
  and $x,z\in V_2$ then 

\[   \m(G)\m(G-\{w,x,y,z\})=\m(G-\{w,x\})\m(G-\{y,z\})+\m(G-\{w,z\})\m(G-\{x,y\}).\]

 \end{theorem}

 \noindent They obtain the following identity
 
 \begin{align*}
  \adj(a,b,c)_{j,k}&\adj(a-1,b,c-1)_{j,k}\\
   =& \adj(a,b,c-1)_{j,k}\adj(a-1,b,c)_{j,k}\\
    &+ \adj(a-1,b+1,c-1)_{j,k}\adj(a,b-1,c)_{j,k}
 \end{align*}
 
 \noindent where $\adj(a,b,c)_{j,k}$ denotes the number of lozenge tilings of a hexagon $H_{a,b,c}$ with opposite side lengths $a,b,c$ with two dents on adjacent sides of length 
 $a$ and $c$ in positions $j$ and $k$ respectively, where $a,b,c,j,k$ are non-negative integers with $1\leq j\leq a$ and $1\leq k\leq c$. 
 
 In their use of Theorem \ref{kj}, they 
 take the graph $G$ to be $\adj(a,b,c)_{j,k}$, but if we take the graph $G$ to be $H_{a,b,c}$ and use Theorem \ref{kk1} with an appropriate choice of labels we obtain the following 
 identity
 
 \begin{align*}
  \hex(a-1,b,c)\adj(a,b,c)_{j,k} &= \hex(a,b,c)\adj(a-1,b,c)_{j,k} \\
   &+ \hex(c,a-1,b+1, c-1, a,b)_k\hex(b-1, c+1, a-1, b,c,a)_j
 \end{align*}

\noindent with the same notations as in Remark \ref{rem1}. Then, Proposition 3 of Ciucu and Fischer \cite{ilse} follows more easily without the need for contigous relations of hypergeometric series that they use in their paper.

\end{remark}

\section{Proofs of the main results}\label{s4}

\begin{proof}[Proof of Theorem \ref{mt1}]

\begin{figure}[!htb]
\centering
\includegraphics[scale=.6]{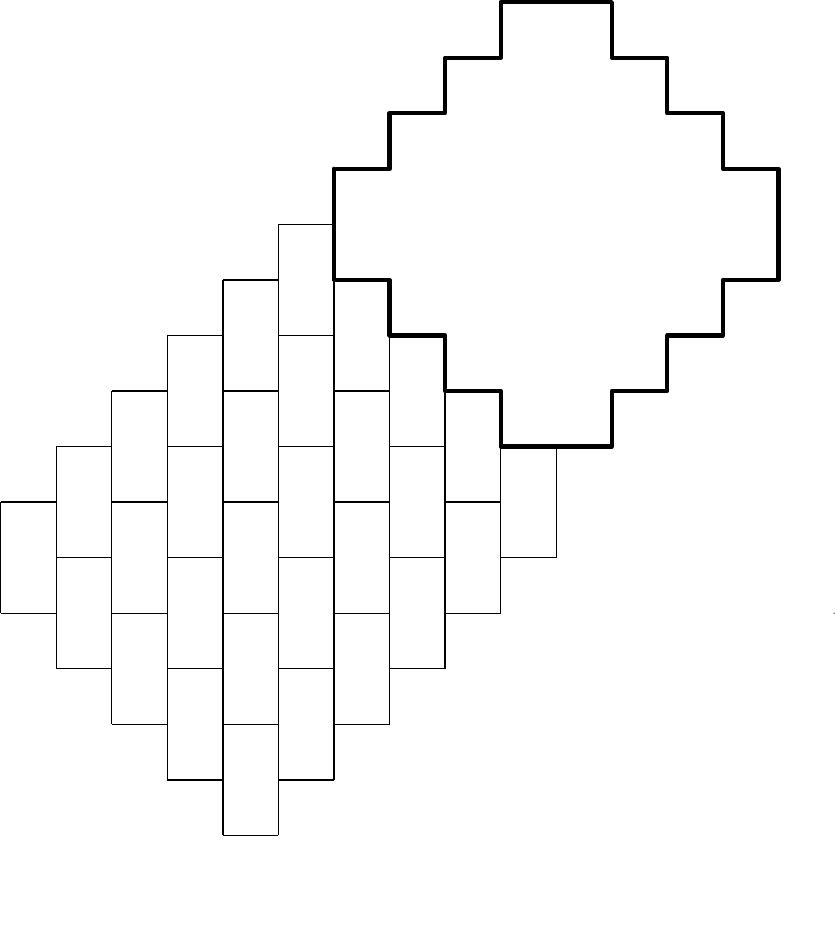}
\caption{Removing the forced dominoes from $\arr$; here $a=5, b=10$, $k=5$}
\label{fig:mt11}
\end{figure}

We shall apply the formula in Theorem \ref{condensation} to the planar dual graph of our region $\arr$, and the vertices $\delta_1, \ldots, \delta_{2n+2k}$. Then the left hand side 
of equation \eqref{ciucu2} becomes the left hand side of equation \eqref{em1}, and the right hand side of equation \eqref{ciucu2} becomes the right hand side of \eqref{em1}. We 
just need to verify that the quantities expressed in equation \eqref{em1} are indeed given by the formulas described in the statement of Theorem \ref{mt1}.

The first statement follows immediately by noting that the added squares on the south eastern side of $\arr$ forces some domino tilings. After removing this forced dominoes we are left 
with an Aztec Diamond of order $a$ as shown in Figure \ref{fig:mt11}, whose number of tilings is given by Theorem \ref{adm}.

\begin{figure}[!htb]
\minipage{0.50\textwidth}
  \includegraphics[scale=0.6, center]{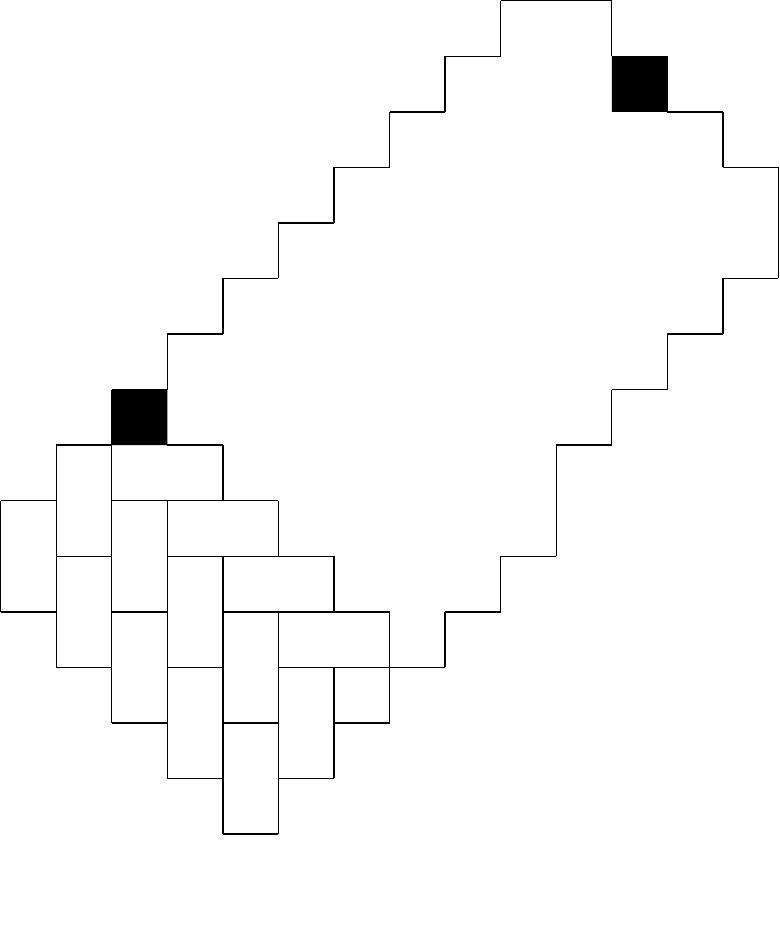}
\endminipage\hfill
\minipage{0.50\textwidth}
  \includegraphics[scale=0.6, center]{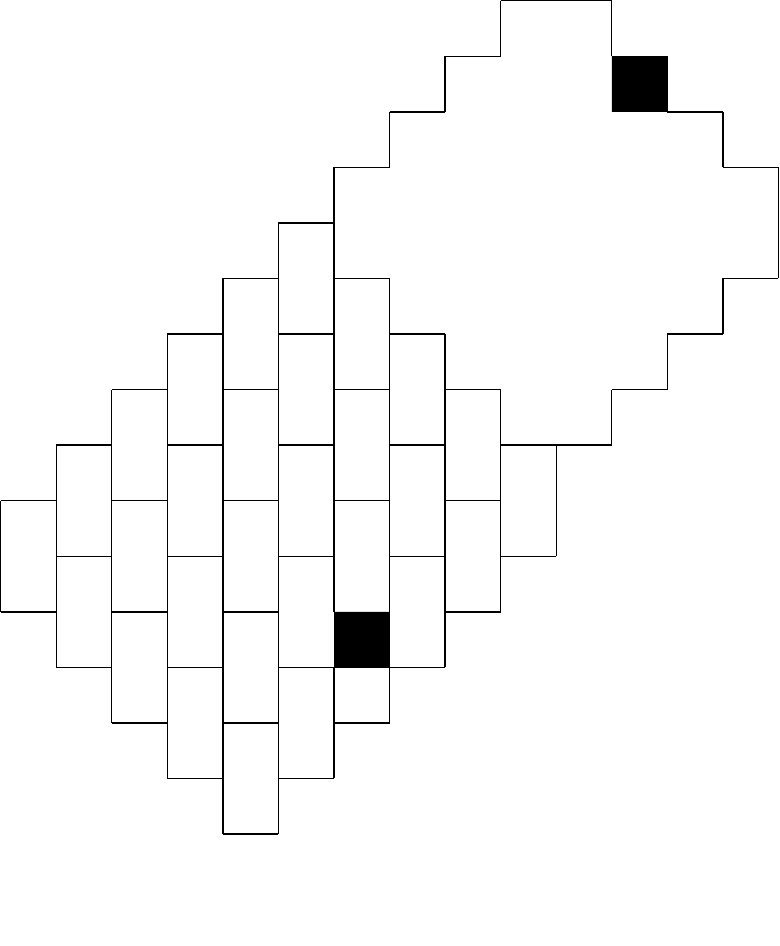}
\endminipage
\caption{Choices of $\al$-defects that lead to no tiling of $\arr$}
\label{fig:mt12}
\end{figure}

\begin{figure}[!htb]
\centering
\includegraphics[scale=.6]{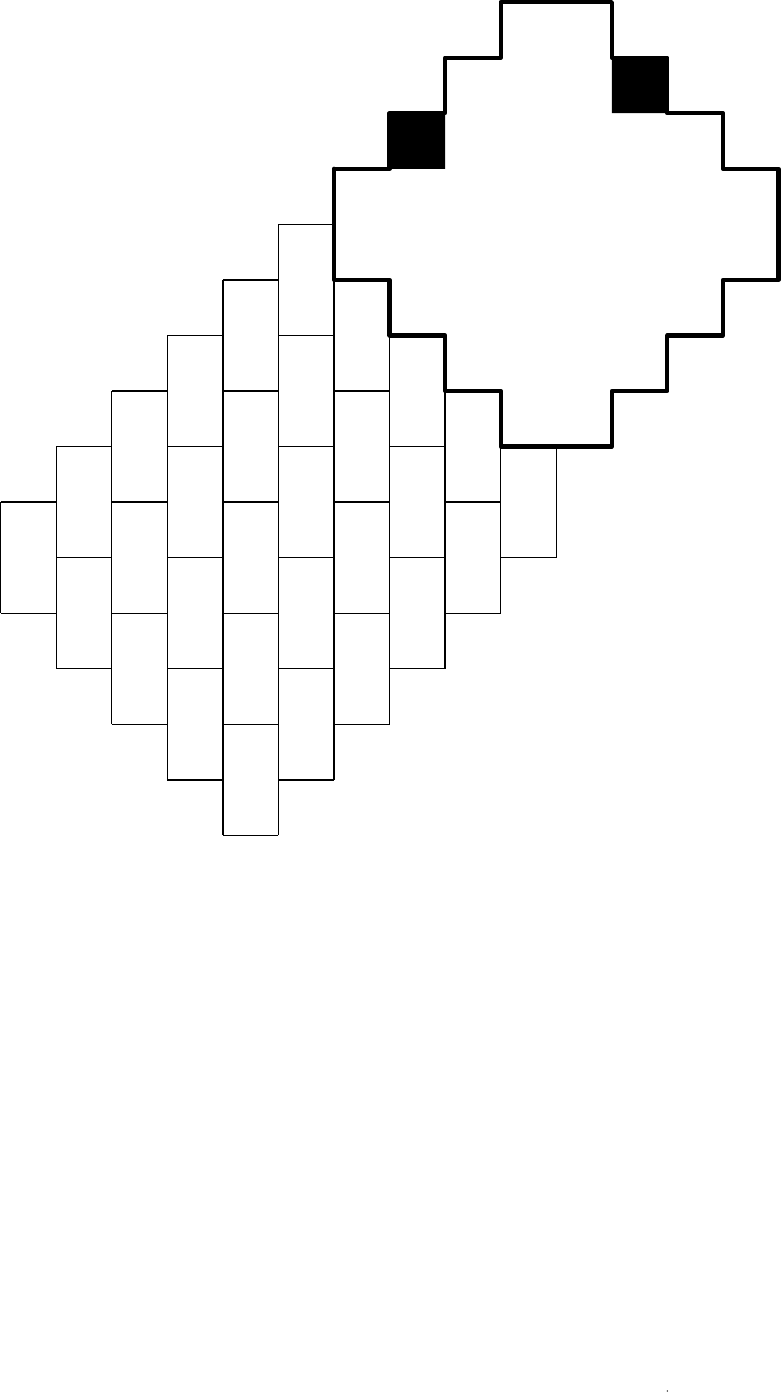}
\caption{Choice of $\al$-defect, not sharing an edge with some $\ga_l$}
\label{fig:mt13}
\end{figure}

The possibilities in the second statement are as follows. If an $\al$ square shares an edge with some $\ga_l$, then the region cannot be covered by any domino as illustrated in the right image of Figure \ref{fig:mt12}. 
Again, if $\al_i$ is on the northwestern side at a distance of atmost $k$ from the western corner, then the strips of forced dominoes along the sourthwestern side interfere with the $\al_i$ and hence there 
cannot be any tiling in this case as illustrated in the left image of Figure \ref{fig:mt12}. If neither of these situation is the case, then due to the squares $\ga_1, \ldots, \ga_k$ on the southeastern side, there are forced dominoes as shown in 
Figure \ref{fig:mt13} and then $\al_i$ and $\be_j$ are defects on an Aztec Diamond on adjacent sides and then the second statement follows from Proposition \ref{ad_i_j}.

\begin{figure}[!htb]
\minipage{0.50\textwidth}
  \includegraphics[scale=0.6, center]{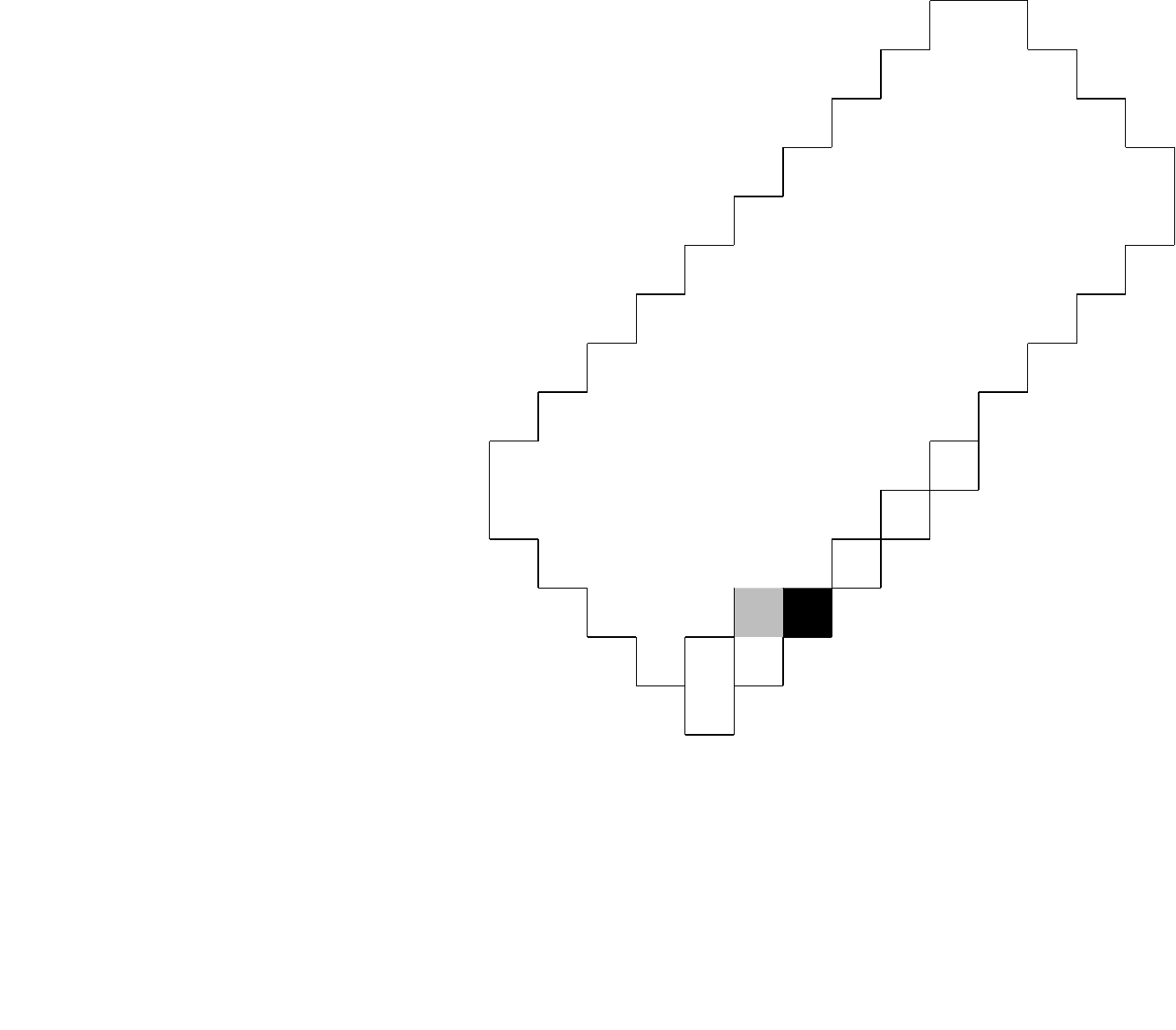}
\endminipage\hfill
\minipage{0.50\textwidth}
  \includegraphics[scale=0.6, center]{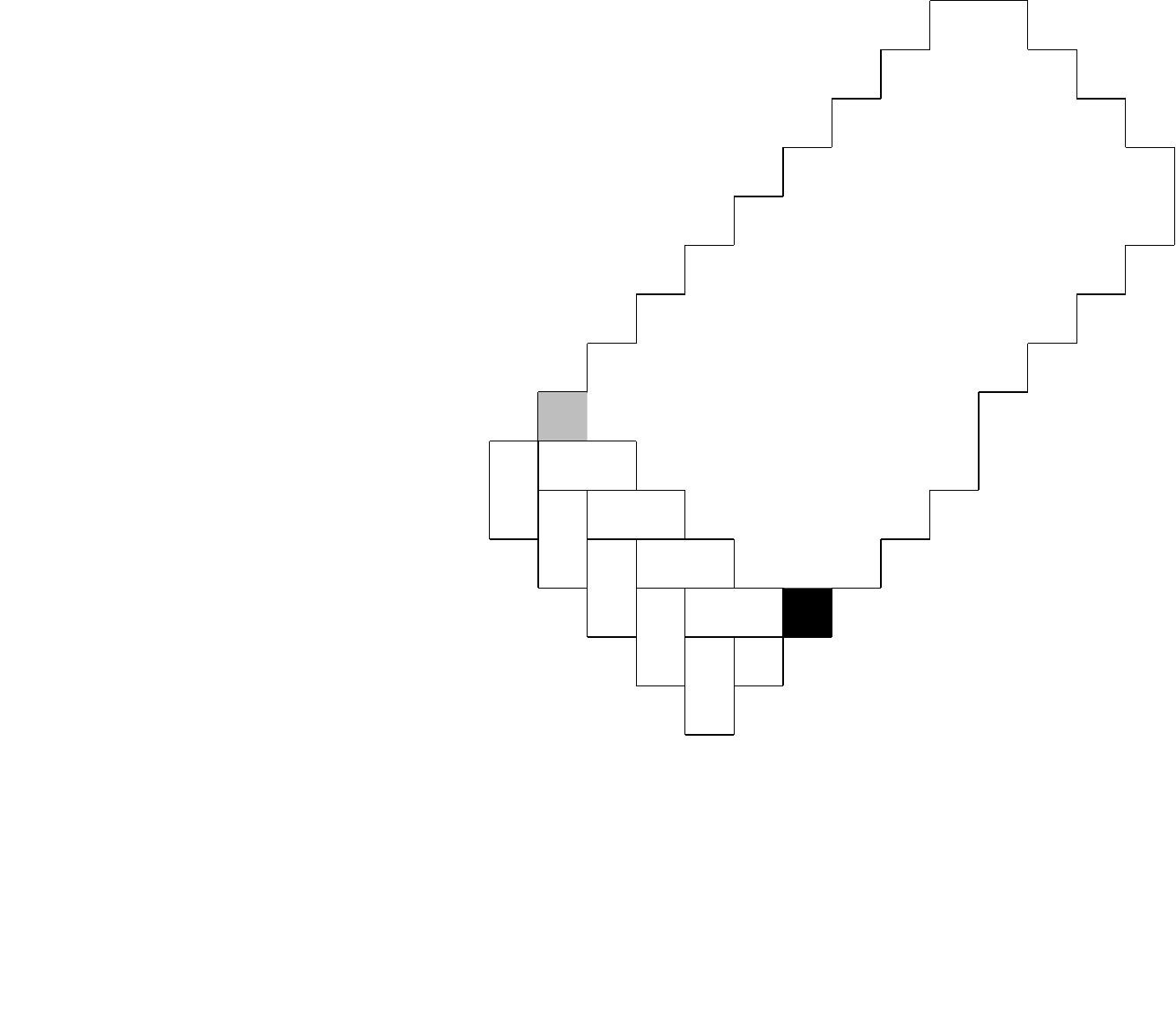}
\endminipage
\caption{Choices of $\al$ and $\ga$-defects that lead to no tiling of $\arr$}
\label{fig:al-ga-no}
\end{figure}

To prove the validity of the third statement, we notice that if an $\al$ and $\ga$ defect share an edge then, there are two possibilities, either the $\al$ defect is above the $\ga$ defect in which 
case we have some forced dominoes as shown in the left of Figure \ref{fig:al-ga-yes} and we are reduced to finding the number of domino tilings of an Aztec Diamond; or the 
$\al$-defect is to the left of a $\ga$-dent, in which case, we get no tilings as shown in the left of Figure \ref{fig:al-ga-no} as the forced dominoes interfere in this case.

If $\al_i$ and $\ga_j$ share no edge in common, then we get no tiling if the $\al$-defect is on the northwestern side at a distance of atmost $k-1$ from the western corner as illustrated 
in the right of Figure \ref{fig:al-ga-no}. If the $\al$-defect is in the northwestern side at a distance more than $k-1$ from the western corner then the situation is as shown in 
the right of Figure \ref{fig:al-ga-yes} and is described in Proposition \ref{ar_k-1_i}. If the $\al$-defect is in the southeastern side then the situation is as shown in the middle of Figure \ref{fig:al-ga-yes} 
and is described in Proposition \ref{ar_k_i}.

\begin{figure}[!htb]
\minipage{0.33\textwidth}
  \includegraphics[scale=0.45, left]{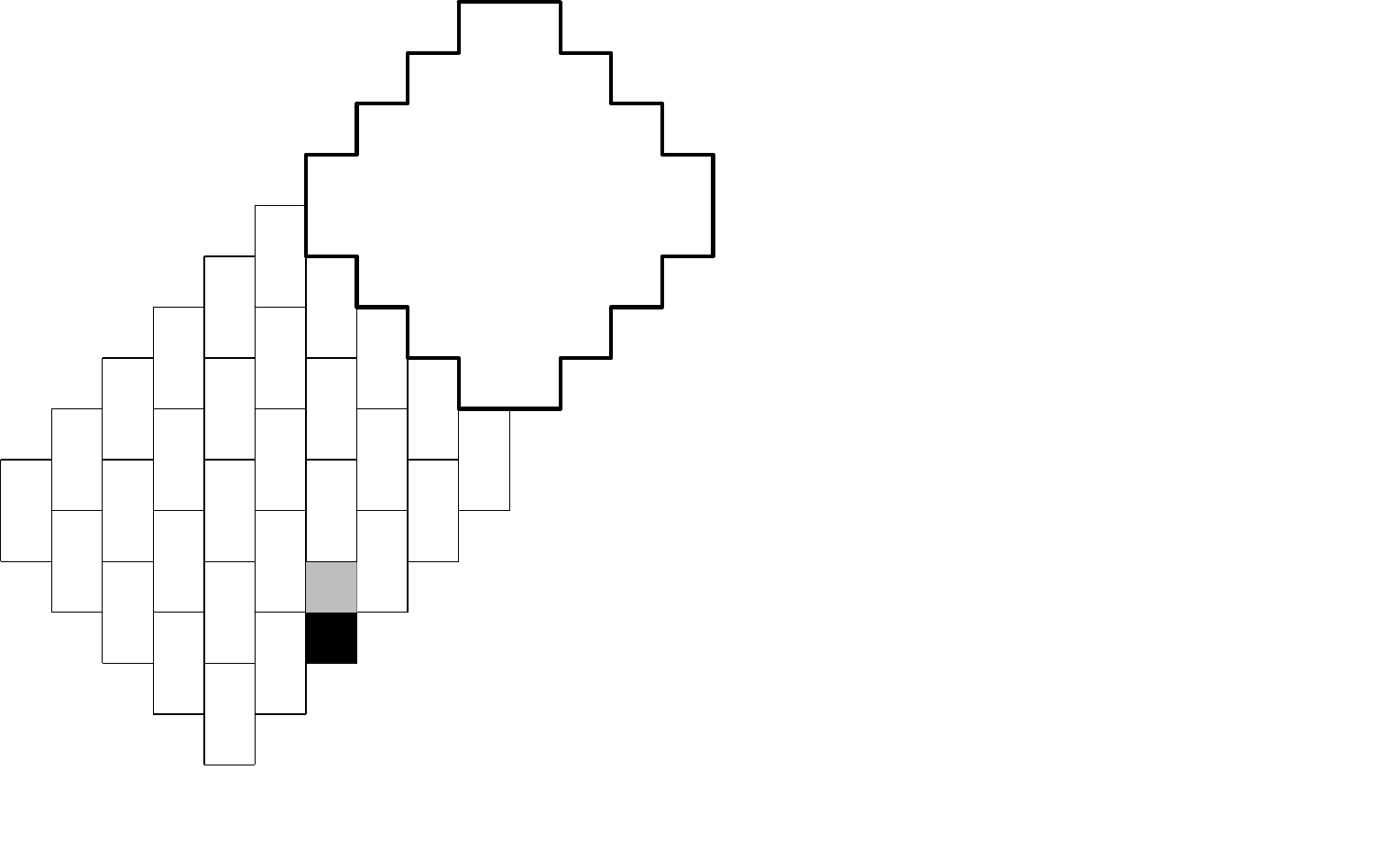}
\endminipage\hfill
\minipage{0.33\textwidth}
  \includegraphics[scale=0.45, center]{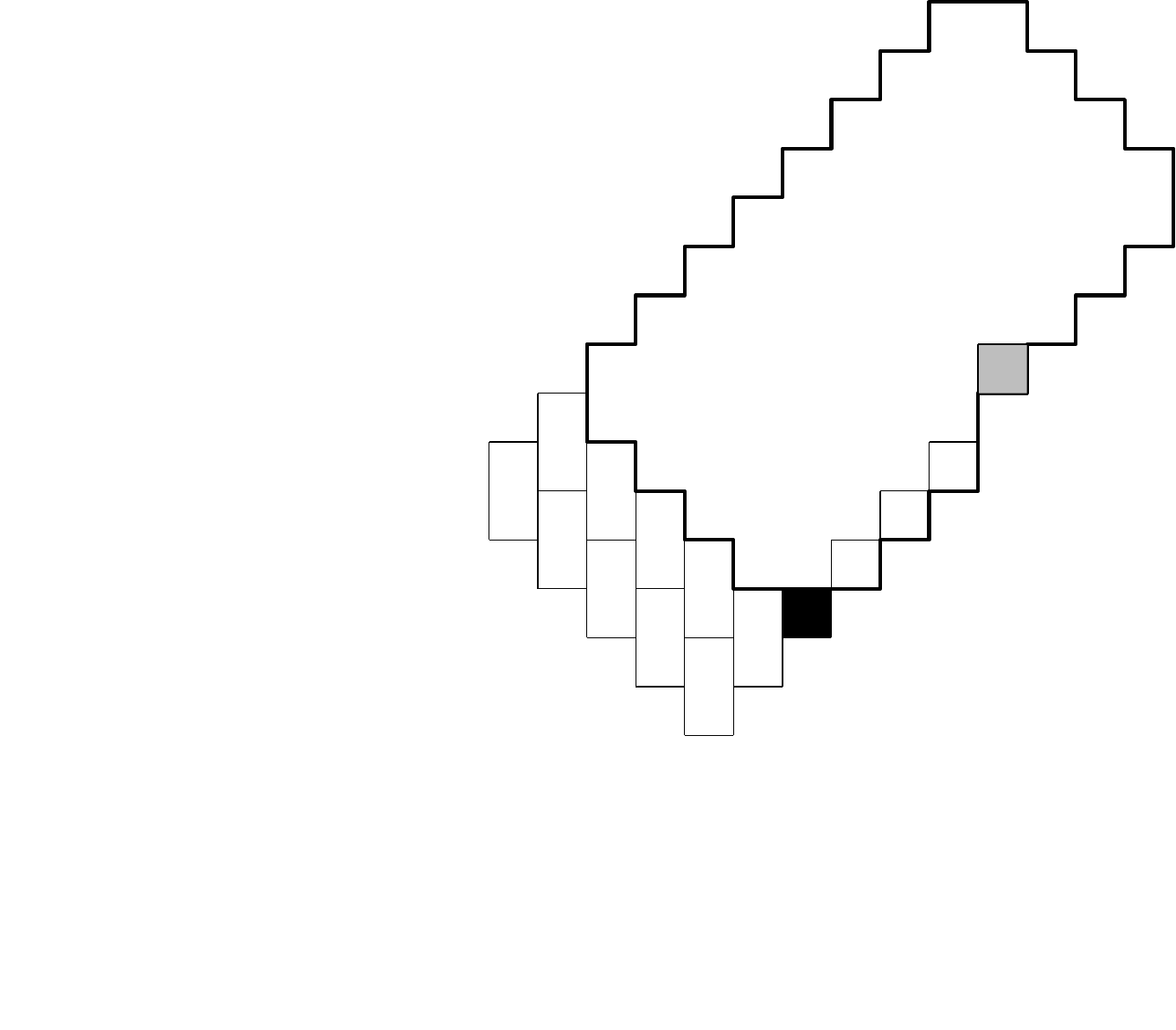}
\endminipage\hfill
\minipage{0.33\textwidth}
  \includegraphics[scale=0.45, right]{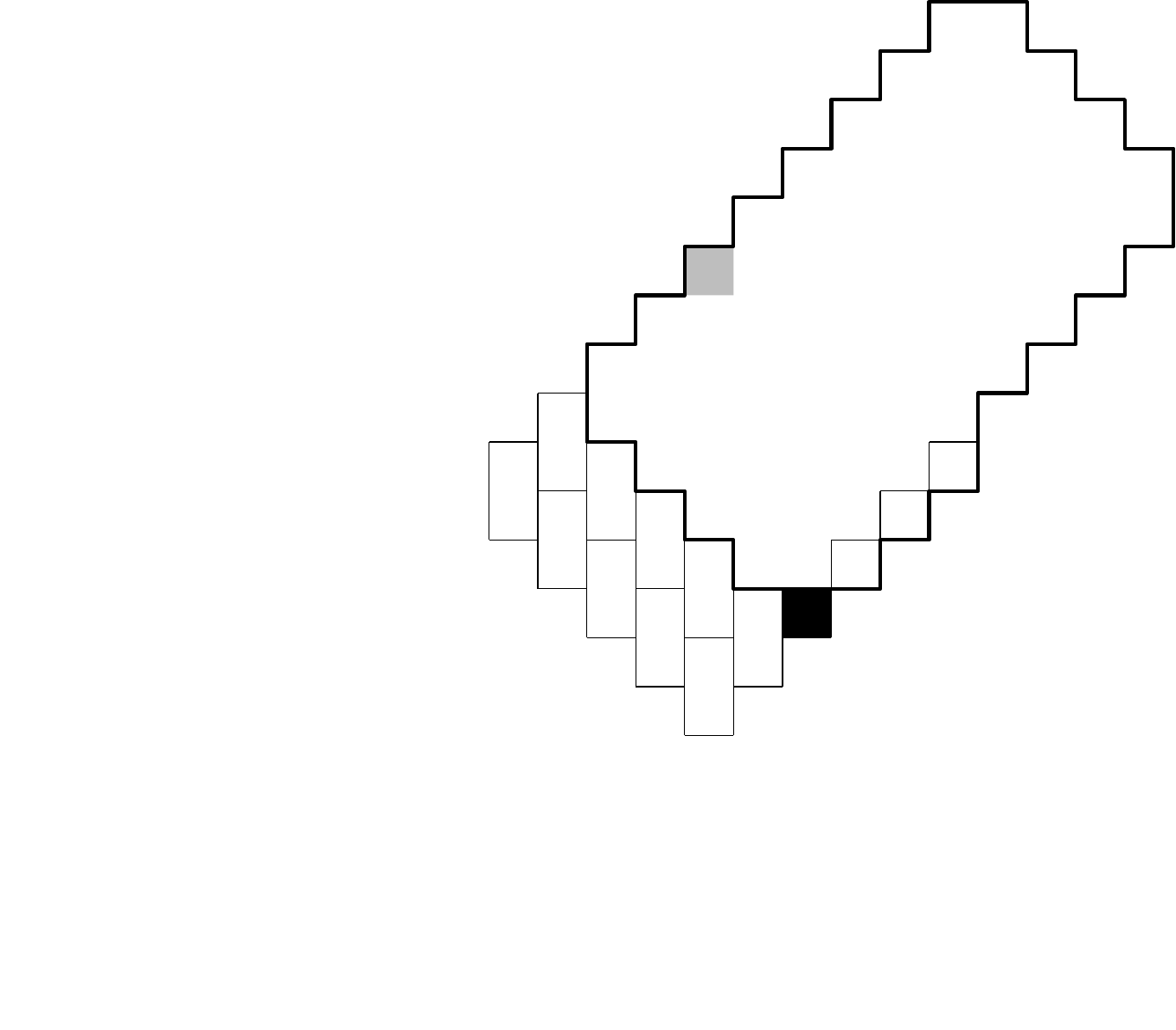}
\endminipage
\caption{Choices of $\al$ and $\ga$-defects that lead to tiling of $\arr$}
\label{fig:al-ga-yes}
\end{figure}

The fourth statement follows immediately from the checkerboard drawing (see Figure \ref{fig:check}) of an Aztec rectangle and the condition that a tiling by dominoes exists for such a board if and only if 
the number of white and black squares are the same. In all other cases, the number of tilings is $0$.

\end{proof}

\begin{proof}[Proof of Theorem \ref{mt2}]
 Let $\ar$ be the region obtained from $\arr$ by removing $k$ of the squares $\al_1, \ldots, \al_{n+k}$. We now apply Theorem \ref{condensation} to the planar dual graph of $\ar$, with the 
 removed squares choosen to be the vertices corresponding to the $n$ $\al_i$'s inside $\ar$ and to $\be_1, \ldots, \be_n$. The left hand side of equation \eqref{ciucu2} is now the 
 required number of tilings and the right hand side of equation \eqref{ciucu2} is the Pfaffian of a $2n\times 2n$ matrix with entries of the form $\m (\ar \setminus \{\al_i, \be_j\})$, 
 where $\al_i$ is not one of the unit squares that we removed from $\arr$ to get $\ar$.
 
 We now notice that $\m (\ar \setminus \{\al_i, \be_j\})$ is an Aztec rectangle with all its defects confined to three of the sides. So, we can apply Theorem \ref{mt1} and it gives us 
 an expression for $\m (\ar \setminus \{\al_i, \be_j\})$ as the Pfaffian of a $(2k+2)\times (2k+2)$ matrix of the type described in the statement of Theorem \ref{mt1}.
\end{proof}

\begin{proof}[Proof of Theorem \ref{mt3}]
 We shall now apply Theorem \ref{condensation} to the planar dual graph of $\ad(a)$ with removed squares choosen to correspond to $\al_1, \ldots, \al_n, \be_1, \ldots, \be_n$. The right hand
  side of equation \eqref{ciucu2} is precisely the right hand side of equation \eqref{emt3}. If $\de_i$ and $\de_j$ are of the same type then $\ad(a)\setminus \{\de_i, \de_j\}$ does 
  not have any tiling as the number of black and white squares in the checkerboard setting of an Aztec Diamond will not be the same (see Figure \ref{fig:check}). Finally, the proof 
  is complete once we note that $\ad(a)\setminus \{\al_i, \be_j\}$ is an Aztec Diamond with two defects removed from adjacent sides for any choice of $\al_i$ and $\be_j$ and is given by 
  Proposition \ref{ad_i_j}.
\end{proof}

\bibliographystyle{amsplain}

\end{document}